\documentclass{scrartcl}

\usepackage[T1]{fontenc}
\usepackage[ansinew]{inputenc}
\usepackage{amsmath, amsthm, amssymb}
\usepackage{dsfont}
\usepackage{graphicx}
\usepackage{subfig}
\usepackage{bm}
\usepackage{eufrak}

\newcommand{\norm}[1]{\left\| #1 \right\|}  
\newcommand{\scprd}[1]{\left\langle #1 \right\rangle}  
\newcommand{\N}{\mathbb{N}}  

\newcommand{\R}{\mathbb{R}}
\newcommand{\C}{\mathbb{C}}

\newcommand{\ran}{\operatorname{ran}}

\newcommand{\ul}{\underline}
\newcommand{\ol}{\overline}

\renewcommand{\Re}{\operatorname{Re}}
\renewcommand{\Im}{\operatorname{Im}}

\numberwithin{equation}{section}

\newtheorem{thm}{Theorem}[section]
\newtheorem{cor}[thm]{Corollary}
\newtheorem{prop}[thm]{Proposition}

\theoremstyle{definition} 
\theoremstyle{definition}

\title{Well-posedness of non-autonomous linear evolution equations for generators whose commutators are scalar}

\author{Jochen Schmid\\  
\small Fachbereich Mathematik, Universität Stuttgart, D-70569 Stuttgart, Germany\\
\small jochen.schmid@mathematik.uni-stuttgart.de}

\date{}

\begin{document}

\maketitle

\begin{abstract}
\small{ \noindent We prove the well-posedness of non-autonomous linear evolution equations for generators $A(t): D(A(t)) \subset X \to X$ whose pairwise commutators are complex scalars and, in addition, we establish an explicit representation formula for the evolution.
We also prove well-posedness in the more general case where instead of the $1$-fold commutators only the $p$-fold commutators of the operators $A(t)$ are complex scalars. All these results are furnished with rather mild stability and regularity assumptions: indeed, stability in $X$ and strong continuity conditions are sufficient. 
Additionally, we improve a well-posedness result of Kato for group generators $A(t)$ by showing that the original norm continuity condition
can be relaxed to strong continuity.  
Applications include Segal field operators and Schrödinger operators for particles in external electric fields.} 
\end{abstract}

{\small \noindent \emph{2010 Mathematics Subject Classification:} 47D06 (primary), 
35Q41 (secondary)   
\\
\emph{Key words and phrases:} well-posedness of non-autonomous linear evolution equations for semigroup generators whose pairwise commutators are complex scalars and for group generators}

\section{Introduction}

In this paper, we are concerned with non-autonomous linear evolution equations 
\begin{align} \label{eq: ivp}
x' = A(t)x  \,\, (t \in [s,1]) \quad \text{and} \quad  x(s) = y
\end{align}
for densely defined linear operators $A(t): D(A(t)) \subset X \to X$  ($t \in [0,1]$) and initial values $y \in Y \subset D(A(s))$ at initial times $s \in [0,1)$. Well-posedness of such evolution equations has been studied by many authors in a large variety of situations. See, for instance, \cite{Pazy83}, \cite{Schnaubelt02}, \cite{NagelNickel02}, 
\cite{NeidhardtZagrebnov09} for an overview.  
In this paper, we are primarily interested in the special situation of semigroup generators $A(t)$ whose first ($1$-fold) or higher ($p$-fold) commutators at distinct times are complex scalars, in short: 
\begin{gather}
[A(t_1), A(t_2)]  = \mu(t_1,t_2) \in \C    \label{eq: 1-facher kommutator skalar}\\
\text{or} \notag \\
\big[ \dots \big[ [A(t_1), A(t_2)], A(t_3) \big] \dots, A(t_{p+1}) \big]  = \mu(t_1, \dots, t_{p+1}) \in \C     \label{eq: p-facher kommutator skalar}
\end{gather}
in some sense to be made precise (see the commutation relations~\eqref{eq: vertrel, p=1}, \eqref{eq: vertrel, p=1, direkter} and~\eqref{eq: vertrel, allg p}, \eqref{eq: vertrel, allg p, direkter}). 
In this special situation we prove well-posedness 
for~\eqref{eq: ivp} on suitable dense subspaces $Y$ of $X$ 
and, moreover, in the case~\eqref{eq: 1-facher kommutator skalar} we prove the representation formula
\begin{align} \label{eq: repr}
U(t,s) = e^{ \ol{\int_s^t A(\tau)\, d\tau } } \, e^{1/2 \int_s^t \int_s^{\tau} \mu(\tau,\sigma)\,d\sigma d\tau}
\end{align}
for the evolution 
generated by the operators $A(t)$. 
We thereby generalize a well-posedness result of Goldstein and of Nickel and Schnaubelt from~\cite{Goldstein69}, \cite{NickelSchnaubelt98} dealing with the special case of~\eqref{eq: 1-facher kommutator skalar} where $\mu \equiv 0$: in \cite{Goldstein69} contraction semigroup generators are considered, while in~\cite{NickelSchnaubelt98} contraction semigroup generators are replaced by 
general semigroup generators and the formula~\eqref{eq: repr} with $\mu \equiv 0$ is proved.

What one gains by 
restricting oneself to the special class of semigroup generators with~\eqref{eq: 1-facher kommutator skalar} or~\eqref{eq: p-facher kommutator skalar} -- instead of considering general semigroup generators as 
in~\cite{Kato53}, \cite{Kato70}, \cite{Kato73}, for instance -- is that well-posedness 
can be established under fairly weak stability and regularity conditions:
1. It is sufficient -- just as in the case of commuting operators from~\cite{Goldstein69}, \cite{NickelSchnaubelt98} -- to require stability of the family $A$ only in $X$. In contrast to the well-posedness theorems from~\cite{Kato70} or~\cite{Kato73}, for instance, it is not necessary to additionally require stability in a suitable invariant and suitably normed dense subspace $Y$ of $X$ contained in all the domains of the operators $A(t)$, which is generally difficult to verify unless the domains of the $A(t)$ are time-independent.
2. It is sufficient -- similarly to the case of commuting operators from~\cite{Goldstein69}, \cite{NickelSchnaubelt98} or to the elementary case of bounded operators -- to require strong continuity conditions: indeed, it is sufficient if 
\begin{align*}
t \mapsto A(t)y \quad \text{and} \quad  (t_1,\dots,t_{k+1}) \mapsto \big[\dots, [[A(t_1), A(t_2)], A(t_3)] \dots, A(t_{k+1}) \big]y 
\end{align*}
are continuous for $k \in \{1, \dots, p\}$ and $y$ in a dense subspace $Y$ of $X$ contained in all the respective domains.
In contrast to the well-posedness theorems from~\cite{Kato70} or~\cite{Kato73}, this subspace $Y$ need not be normed in any way whatsoever and $t \mapsto A(t)|_Y$ need not be norm continuous. And furthermore, it is not necessary to require an additional $W^{1,1}$-regularity condition on certain auxiliary operators $S(t): Y \to X$ (as in the well-posedness theorems from~\cite{Kato70}, \cite{Kato73} for general semigroup generators $A(t)$) or an additional regularity condition on certain auxiliary norms $\norm{\,.\,}_t^{\pm}$ on $Y$ (as in the special well-posedness result from~\cite{Kato70} for a certain kind of group generators). 
Such additional regularity conditions are necessary for well-posedness in general situations without commutator conditions 
of the kind~\eqref{eq: 1-facher kommutator skalar} or~\eqref{eq: p-facher kommutator skalar} 
-- even if the domains of the $A(t)$ are time-independent (see the examples in~\cite{Phillips53}, \cite{EngelNagel00} or~\cite{SchmidGriesemer15},
for instance). 
%
%

As is well-known from~\cite{Magnus54}, \cite{Fer58}, \cite{Wilcox67}, in the 
case of bounded operators $A(t)$ one has representation formulas of Campbell--Baker--Hausdorff and Zassenhaus type for the evolution, 
which in the case~\eqref{eq: 1-facher kommutator skalar} reduce to our representation formula~\eqref{eq: repr}. It should be noticed, however, that for bounded operators condition~\eqref{eq: 1-facher kommutator skalar} can be satisfied only if $\mu \equiv 0$, so that~\eqref{eq: repr} is independent of~\cite{Magnus54}, \cite{Fer58}, \cite{Wilcox67} (for non-zero $\mu$). In view of the representation formulas from~\cite{Wilcox67} it is desirable to prove representation formulas analogous to~\eqref{eq: repr} also in the case~\eqref{eq: p-facher kommutator skalar}, but this is left to future research.

All proofs in connection with the special situations
\eqref{eq: 1-facher kommutator skalar} or~\eqref{eq: p-facher kommutator skalar} are, in essence, based upon the observation 
that in these 
situations the operators $A(r)$ can be commuted -- up to controllable errors -- through the exponential factors of the standard approximants $U_n(t,s)$ from~\cite{Goldstein69}, \cite{Kato70}, \cite{Kato73}, \cite{NickelSchnaubelt98} for the sought evolution, which are of the form
\begin{align*}
U_n(t,s) = e^{A(r_m) \tau_m} \dotsb e^{A(r_1) \tau_1}
\end{align*}
with partition points $r_1, \dots, r_m$ of the interval $[s,t]$. See~\eqref{eq: central} and~\eqref{eq: vorbeiziehrel 2, allg p} respectively.

Apart from proving well-posedness for semigroup generators with~\eqref{eq: 1-facher kommutator skalar} or~\eqref{eq: p-facher kommutator skalar} (which is our primary interest), we also improve the above-mentioned 
special well-posedness result from~\cite{Kato70} for a certain kind of group generators: 
in the spirit of~\cite{Kobayasi79} we show that strong (instead of norm) continuity is sufficient in this result -- just like in our other well-posedness results 
for the case~\eqref{eq: 1-facher kommutator skalar} or~\eqref{eq: p-facher kommutator skalar}. 
And in a certain special case involving quasicontraction group generators with time-independent domains in a uniformly convex space, these other results can also be obtained by applying the improved well-posedness result for group generators.

In Section~2 we state and prove our abstract well-posedness results, Section~2.1 and Section~2.2 being devoted to the case~\eqref{eq: 1-facher kommutator skalar} and~\eqref{eq: p-facher kommutator skalar} respectively and Section~2.3 being devoted to the improved well-posedness result for group generators. Section~2.4 discusses, among other things, the relation of our well-posedness results from Section~2.1 and~2.2 to the results from~\cite{Kato70}, \cite{Kato73}, \cite{Kobayasi79} and to the result from Section~2.3. 
In Section~3 we give some applications, namely to Segal field operators $\Phi(f_t)$ as well as to the related operators $H_{\omega} + \Phi(f_t)$ describing a classical particle coupled to a time-dependent quantized field of bosons (Section~3.1) and finally to Schrödinger operators describing a quantum particle coupled to a time-dependent spatially constant electric field (Section~3.2).

\section{Abstract well-posedness results}

We will use the notion of well-posedness and evolution systems from~\cite{EngelNagel00}. 
So, if $A(t): D(A(t)) \subset X \to X$ for every $t \in I := [0,1]$ is a 
linear operator and $Y$ is a dense subspace of $X$ contained in $\cap_{\tau \in I} D(A(\tau))$, then the initial value problems~\eqref{eq: ivp} for $A$ are called \emph{well-posed on $Y$} if and only if there exists an \emph{evolution system $U$ solving~\eqref{eq: ivp} on $Y$} or, for short, an \emph{evolution system $U$ for $A$ on $Y$}. An evolution system for $A$ on $Y$ is, by definition, a family of bounded operators $U(t,s)$ in $X$ for $s \le t$ such that
\begin{itemize}
\item[(i)] $[s,1] \ni t \mapsto U(t,s)y$ for $y \in Y$ and $s \in [0,1)$ is a continuously differentiable solution of~\eqref{eq: ivp} with values in $Y$,
\item[(ii)] $U(t,s) U(s,r) = U(t,r)$ for all $r \le s \le t$ and $(s,t) \mapsto U(t,s)$ is strongly continuous.
\end{itemize}
Such an evolution system is necessarily unique: if $U$ and $V$ are two evolution systems for $A$ on $Y$, then $[s,t] \ni \tau \mapsto U(t,\tau)V(\tau,s)y$ for every $y \in Y$ is continuous and right differentiable with vanishing right derivative, because 
\begin{align*}
\frac 1 h \big( U(t,\tau+h) - U(t,\tau) \big)z 
= -U(t,\tau+h) \frac{ U(\tau+h,\tau)z - z}{h} \longrightarrow -U(t,\tau)A(\tau)z 
\end{align*}
as $h \searrow 0$ for all $z \in Y$ and because $\frac 1 h (V(\tau+h,s)y - V(\tau,s)y) \longrightarrow A(\tau)V(\tau,s)y$ as $h \to 0$ and $V(\tau,s)y \in Y$ for all $y \in Y$. 
With the help of Corollary~2.1.2 of~\cite{Pazy83} it then follows that
\begin{align*}
V(t,s)y - U(t,s)y = U(t,\tau)V(\tau,s)y \big|_{\tau=s}^{\tau=t} = 0,
\end{align*}
which by the density of $Y$ in $X$ implies $U = V$, as desired.
%
At some places we will also use the notion of 
evolution systems from~\cite{NickelSchnaubelt98}, which is slightly weaker than the one above 
in that it does not require that $[s,1] \ni t \mapsto U(t,s)y$ have values in $Y$ for every $y \in Y$ and $s \in [0,1)$ (while all other conditions from above are taken over). 
We will then speak of \emph{evolution systems in the wide sense for $A$ on $Y$} 
and, in case there exists exactly one such evolution system in the wide sense, we will speak of \emph{well-posedness in the wide sense on $Y$}.
%
%
Commutators of possibly unbounded operators are taken in the operator-theoretic sense, 
\begin{align*}
D([A,B]) := D(AB - BA) = D(AB) \cap D(BA),
\end{align*}
except in some formal heuristic computations (whose formal character will always be pointed out).
$X = (X,\norm{\,.\,})$ will always stand for a complex Banach space, $I = [0,1]$ denotes the compact unit interval, and $\Delta$ the triangle $\{(s,t) \in I^2: s \le t\}$.
We will finally also need the standard notions 
of $(M, \omega)$-stability, of the part of an operator $A$ in a subspace $Y$, and of $A$-admissible subspaces from~\cite{Kato70} or~\cite{Pazy83}, 
and we briefly recall them here for the sake of convenience. A family $A$ of semigroup generators $A(t)$ on $X$ for $t \in I$ is called \emph{$(M,\omega)$-stable} (where 
$M \in [1,\infty)$ and $\omega \in \R$) if and only if
\begin{align*}
\norm{ e^{A(t_n)s_n} \dotsb e^{A(t_1)s_1} } \le M e^{\omega( s_1 + \dotsb + s_n)}
\end{align*}  
for all $n \in \N$ and all $t_1, \dots, t_n \in I$ with $t_n \ge \dotsb \ge t_1$ and $s_1, \dots, s_n \in [0,\infty)$. 
If $A$ is an arbitrary operator in $X$ and $Y$ is an arbitrary subspace of $X$, then the operator $\tilde{A}$ defined by 
\begin{align*}
D(\tilde{A}) := \big\{ y \in D(A) \cap Y: A y \in Y \big\} \quad \text{and} \quad \tilde{A}y := Ay \quad (y \in D(\tilde{A}))
\end{align*} 
is called the \emph{part of $A$ in $Y$} or, for short, the \emph{$Y$-part of $A$}.
%
If $A$ is a semigroup generator on $X$, then a subspace $Y$ of $X = (X,\norm{\,.\,})$ endowed with a norm $\norm{\,.\,}_*$ is called \emph{$A$-admissible} if and only if 
\begin{itemize}
\item[(i)] $(Y,\norm{\,.\,}_*)$ is a Banach space densely and continuously embedded in $(X,\norm{\,.\,})$, 
\item[(ii)] $e^{A s} \, Y \subset Y$ for all $s \in [0,\infty)$ and the restriction 
$e^{A\,.\,}|_Y$ is a strongly continuous semigroup in $(Y,\norm{\,.\,}_*)$. 
\end{itemize}
In this case, the semigroup 
$e^{A\,.\,}|_Y$ is generated by the part $\tilde{A}$ of $A$ in $Y$.
%
%
%

\subsection{Scalar $1$-fold commutators}

In this subsection we prove well-posedness for~\eqref{eq: ivp} in the case~\eqref{eq: 1-facher kommutator skalar} where the $1$-fold commutators of the operators $A(t)$ are complex scalars. We have to make precise the merely formal commutation relation~\eqref{eq: 1-facher kommutator skalar}, of course, and we begin with a well-posedness result where~\eqref{eq: 1-facher kommutator skalar} is replaced by the formally equivalent commutation relation~\eqref{eq: vertrel, p=1} for the semigroups 
$e^{A(t)\,.\,}$ with the generators $A(s)$. 
In addition to well-posedness this theorem also yields a representation formula for the evolution. It is a generalization of a well-posedness result of Goldstein~\cite{Goldstein69} (Theorem~1.1) and -- after the slight modifications discussed in~\eqref{eq: stab nickelschnaubelt} and~\eqref{eq: stab nickelschnaubelt, strikt} below -- of Nickel and Schnaubelt~\cite{NickelSchnaubelt98} (Theorem~2.3 and Proposition~2.5). 

\begin{thm} \label{thm: wohlgestelltheit, p=1}
Suppose $A(t): D(A(t)) \subset X \to X$ for every $t \in I$ is the generator of 
a strongly continuous semigroup on $X$ such that $A$ is $(M,\omega)$-stable for some $M \in [1,\infty)$ and $\omega \in \R$ and 
such that for some complex numbers $\mu(s,t) \in \C$
\begin{align} \label{eq: vertrel, p=1}
A(s) e^{A(t)\tau} \supset e^{A(t)\tau} \big( A(s) + \mu(s,t) \tau \big)
\end{align}
for all $s, t \in I$ and $\tau \in [0,\infty)$. Suppose further that the maximal continuity subspace 
\begin{align} \label{eq: max stetur, p=1}
Y^{\circ} := \{ y \in \cap_{\tau \in I} D(A(\tau)): t \mapsto A(t)y \text{ is continuous} \}
\end{align}
is dense in $X$ and that $(s,t) \mapsto \mu(s,t)$ is continuous. Then there exists a unique evolution system $U$ for $A$ on $Y^{\circ}$ and it is given by
\begin{align*}
U(t,s) = e^{ \ol{ ( \int_s^t A(\tau) \, d\tau )^{\circ} } } e^{1/2 \int_s^t \int_s^{\tau} \mu(\tau, \sigma) \, d\sigma \, d\tau } \quad ((s,t) \in \Delta),
\end{align*}
where $\big( \int_s^t A(\tau) \, d\tau \big)^{\circ}$ is the (closable) operator defined by $y \mapsto \int_s^t A(\tau)y \, d\tau$ on $Y^{\circ}$.
\end{thm}

\begin{proof}
(i) We first show, in three steps, the existence of an evolution system $U$ for $A$ on~$Y^{\circ}$, which is then necessarily unique by the remark at the beginning of Section~2. 
In order to do so we approximate the sought evolution $U$ by the standard approximants $U_n$ from hyperbolic evolution equations theory, that is, we choose partitions 
\begin{align*}
\pi_n = \{ r_{n \, i}: i \in \{0, \dots, m_n \} \}
\end{align*}
of $I$ with $\operatorname{mesh}(\pi_n) \longrightarrow 0$ as $n \to \infty$ and, for any such partition, we evolve piecewise according to the values of $t \mapsto A(t)$ at the finitely many partition points of $\pi_n$. So,  
\begin{align}  \label{eq: def U_n, 1}
U_n(t,s) := e^{A(r_n(t))  (t-s)} 
\end{align}
for $(s,t) \in \Delta$ with $s$, $t$ lying in the same partition subinterval of $\pi_n$ 
and
\begin{align} \label{eq: def U_n, 2}
U_n(t,s) := e^{A(r_n(t)) (t-r_n(t))} e^{A(r_n^-(t)) (r_n(t)-r_n^-(t))} \dotsb e^{A(r_n(s)) (r_n^+(s)-s)} 
\end{align}
for $(s,t) \in \Delta$ with $s$, $t$ lying in different partition subintervals of $\pi_n$. 
In the equations above,  $r_n(u)$ for $u \in I$ denotes the largest partition point of $\pi_n$ less than or equal to $u$ and $r_n^-(u)$, $r_n^+(u)$ is the neighboring partition point below or above $r_n(u)$, respectively.
\smallskip

We then obtain, by repeatedly applying the assumed commutation relation~\eqref{eq: vertrel, p=1}, the following important commutation relation which allows us to take $A(r)$ from the left of $U_n(t,s)$ to the right and which is central to the entire proof: 
\begin{gather} \label{eq: central}
A(r) U_n(t,s)y = U_n(t,s) \Big( A(r) + \int_s^t \mu(r,r_n(\sigma)) \, d\sigma \Big) y  
\end{gather}
for all $y \in D(A(r))$. 
As a first step, we observe that 
\begin{align} \label{eq: elem eigenschaften U_n}
U_n(t,s)U_n(s,r) = U_n(t,r) \quad \text{and} \quad \norm{U_n(t,s)} \le M e^{\omega (t-s)}
\end{align}
for all $(s,t), (r,s) \in \Delta$ and that $\Delta \ni (s,t) \mapsto U_n(t,s)$ is strongly continuous.
\smallskip

As a second step, we show that $(U_n(t,s)x)$ for every $x \in X$ is a Cauchy sequence in $X$ uniformly in $(s,t) \in \Delta$. 
Since $\cap_{r' \in I} D(A(r'))$ is invariant under the semigroups $e^{A(r)\,.\,}$ for all $r \in I$ by~\eqref{eq: vertrel, p=1}, it follows that $[s,t] \ni \tau \mapsto U_m(t,\tau)U_n(\tau,s)y$ for every $y \in \cap_{r' \in I} D(A(r'))$ is piecewise continuously differentiable (with the partition points of $\pi_m \cup \pi_n$ as exceptional points) and therefore
\begin{align*}
&U_n(t,s)y - U_m(t,s)y = U_m(t,\tau)U_n(\tau,s)y \Big|_{\tau=s}^{\tau=t} \\
&\qquad = \int_s^t U_m(t,\tau) \big( A(r_n(\tau)) - A(r_m(\tau)) \big) U_n(\tau,s)y \, d\tau 
= \int_s^t U_m(t,\tau) U_n(\tau,s) \\ 
&\qquad \qquad \qquad \quad \Big(   A(r_n(\tau)) - A(r_m(\tau)) +  \int_s^{\tau} \mu(r_n(\tau), r_n(\sigma)) - \mu(r_m(\tau), r_n(\sigma)) \, d\sigma  \Big)y \, d\tau
\end{align*}
for every $y \in \cap_{r' \in I} D(A(r'))$ where, for the last equation, \eqref{eq: central} has been used. So, 
\begin{align*}
&\sup_{(s,t) \in \Delta} \norm{ U_n(t,s)y - U_m(t,s)y } \le M^2 e^{w(b-a)} \bigg( \int_a^b \norm{ A(r_n(\tau))y - A(r_m(\tau))y } \, d\tau  \\
&\qquad + \int_a^b \int_a^b \big|  \mu(r_n(\tau), r_n(\sigma)) - \mu(r_m(\tau), r_n(\sigma))  \big| \norm{y} \,d\sigma \, d\tau \bigg) 
\longrightarrow 0 \quad (m,n \to \infty)
\end{align*}
for every $y \in Y^{\circ}$ by the uniform continuity of $\tau \mapsto A(\tau)y$ and $(\tau,\sigma) \mapsto \mu(\tau,\sigma)$. And by~\eqref{eq: elem eigenschaften U_n} this uniform Cauchy property extends to all $y \in X$. Consequently, 
\begin{align*}
U(t,s)x := \lim_{n \to \infty} U_n(t,s)x 
\end{align*}
for every $x \in X$ exists uniformly in $(s,t) \in \Delta$ and hence the properties observed in the first step carry over from $U_n$ to $U$.
\smallskip

As a third step, we show that $t \mapsto U(t,s)y$ for every $y \in Y^{\circ}$ is a continuously differentiable solution to~\eqref{eq: ivp} with values in $Y^{\circ}$. 
Since $\tau \mapsto U_n(\tau,s)y$ for $y \in \cap_{r' \in I} D(A(r'))$ is piecewise continuously differentiable with piecewise derivative
\begin{align*}
[s,t] \setminus \pi_n \ni \tau \mapsto A(r_n(\tau)) U_n(\tau,s)y 
= U_n(\tau,s)\Big( A(r_n(\tau)) + \int_s^{\tau} \mu(r_n(\tau) , r_n(\sigma)) \, d\sigma \Big)y
\end{align*}
by virtue of~\eqref{eq: central}, we have
\begin{align*}
U_n(t,s)y = y + \int_s^t  U_n(\tau,s)\Big( A(r_n(\tau)) + \int_s^{\tau} \mu(r_n(\tau) , r_n(\sigma)) \, d\sigma \Big)y \, d\tau
\end{align*} 
and therefore
\begin{align*}
U(t,s)y = y + \int_s^t U(\tau,s) \Big( A(\tau) + \int_s^{\tau} \mu(\tau,\sigma) \,d\sigma \Big) y \, d\tau
\end{align*} 
for all $y \in Y^{\circ}$. 
So, $t \mapsto U(t,s)y$ is continuously differentiable for every $y \in Y^{\circ}$ 
with derivative
\begin{align*}
t \mapsto U(t,s) \Big( A(t) + \int_s^{t} \mu(t,\sigma) \,d\sigma \Big) y = \lim_{n\to \infty} A(t)U_n(t,s)y = A(t)U(t,s)y,
\end{align*}
where the last two equations hold by~\eqref{eq: central} and the closedness of $A(t)$. Also, since for all $y \in Y^{\circ}$ and $r \in I$
\begin{align*}
A(r) U_n(t,s)y \longrightarrow U(t,s) \Big( A(r) + \int_s^{t} \mu(r,\sigma) \,d\sigma \Big) y \quad (n \to \infty),
\end{align*}
we see by the closedness of the operators $A(r)$ that $U(t,s)y \in Y^{\circ}$ for $y \in Y^{\circ}$. So, in summary, we have shown that $U$ is an evolution system for $A$ on $Y^{\circ}$.
\smallskip

(ii) We now show, in three steps, that $\big( \int_s^t A(\tau) \, d\tau \big)^{\circ}$ for every fixed $(s,t) \in \Delta$ is closable and that its closure generates a strongly continuous semigroup in $X$ with
\begin{align*}
e^{ \ol{ ( \int_s^t A(\tau) \, d\tau )^{\circ} } }  = U(t,s) e^{-1/2 \int_s^t \int_s^{\tau} \mu(\tau, \sigma) \, d\sigma \, d\tau }.
\end{align*}
As a first step, we show a discrete version of the above representation formula:  
more precisely, we show that $B_n := \int_s^t A(r_n(\tau)) \,d\tau$ is closable and that $B_n$ generates a strongly continuous semigroup with the following decomposition of Zassenhaus type: 
\begin{align} \label{eq: thm 2.1 zwbeh (ii), 1} 
e^{\ol{B_n} r} = U_n^r(t,s) e^{-1/2 (\int_s^t \int_s^{\tau} \mu(r_n(\tau), r_n(\sigma)) \, d\sigma \, d\tau) r^2 } \quad (r \in [0,\infty)),
\end{align}
where the operators $U_n^r(t,s)$ are defined in the same way as the operators $U_n(t,s)$ above 
with the only difference that now the generators $A(u)$ are all multiplied by the number $r$. Indeed, by the assumed commutation relations, we obtain the following commutation relations for semigroups,
\begin{align} \label{eq: vertrel halbgr}
e^{A_i \sigma} e^{A_j \tau} = e^{A_j \tau} e^{A_i \sigma} e^{\mu_{i j} \sigma \tau} \quad (\sigma, \tau \in [0,\infty)),
\end{align}
where $A_k := A(t_k)h_k$ and $\mu_{k l} := \mu(t_k, t_l)h_k h_l$ for arbitrary $t_k, t_l \in I$ and $h_k, h_l \in [0,\infty)$. (In fact, if $y \in D(A_i)$, then 
\begin{align*}
e^{A_j \tau} e^{A_i \sigma} e^{\mu_{i j} \sigma \tau} y - e^{A_i \sigma} e^{A_j \tau} y = e^{A_i (\sigma-r)} e^{A_j \tau} e^{A_i r} e^{\mu_{i j} r \tau} y \big|_{r=0}^{r=\sigma}
\end{align*}
and $[0,\sigma] \ni r \mapsto e^{A_i (\sigma-r)} e^{A_j \tau} e^{A_i r} e^{\mu_{i j} r \tau} y$ is differentiable with derivative $0$.) 
With the help of~\eqref{eq: vertrel halbgr} one verifies that 
\begin{align} \label{eq: thm 2.1 halbgr}
[0,\infty) \ni r \mapsto e^{A_m r} \dotsb e^{A_1 r} e^{-1/2 \sum_{i \le j} \mu_{j i} r^2}
\end{align}
is a strongly continuous semigroup in $X$. As this semigroup, by the assumed commutation relation, leaves the subspace $D(A_1) \cap \dotsb \cap D(A_m)$ invariant, its generator contains the operator $A_1 + \dotsb + A_m$, which is therefore closable with closure equal to the generator. Since $B_n$ is 
of the form $A_1 + \dotsb + A_m$ and since the right-hand side of~\eqref{eq: thm 2.1 zwbeh (ii), 1} is of the form~\eqref{eq: thm 2.1 halbgr} (because $\mu_{i i} = 0$ by virtue of~\eqref{eq: vertrel halbgr}), the assertion of the first step follows.
\smallskip

As a second step, we observe that the limit $T(r)x := \lim_{n \to \infty} e^{ \ol{B_n} r } x$ exists locally uniformly in $r \in [0,\infty)$ for every $x \in X$ and that $T$ is a strongly continuous semigroup in $X$. Indeed, with the same arguments as in~(i), it follows that $(U_n^r(t,s)x)$ is convergent locally uniformly in $r$ for every $x \in X$ with limit denoted by $U^{r}(t,s)x$ and therefore the strongly continuous semigroups $e^{\ol{B_n} \,.\,}$ by~\eqref{eq: thm 2.1 zwbeh (ii), 1} are strongly convergent locally uniformly in $r$, so that
\begin{align} \label{eq: def T}
T(r)x := \lim_{n \to \infty} e^{\ol{B_n} r} x = U^r(t,s) e^{-1/2 ( \int_s^t \int_s^{\tau} \mu(\tau, \sigma) \,d\sigma \,d\tau ) r^2 } x \quad (x \in X)
\end{align}
defines a strongly continuous semigroup $T$ on $X$.   
\smallskip

As a third step, we show that the generator $A_T$ of this semigroup is given by $\ol{B^{\circ}}$ where $B^{\circ} := \big( \int_s^t A(\tau) \,d\tau \big)^{\circ}$, from which the desired representation formula for $U$ then follows by~\eqref{eq: def T} (because $U^1(t,s) = U(t,s)$).
Indeed, for all $y \in Y^{\circ}$, 
\begin{align*}
\frac{T(h)y - y}{h} = \lim_{n \to \infty} \frac{e^{\ol{B_n} h} y - y}{h} &= \lim_{n \to \infty} \frac{1}{h} \int_0^h e^{\ol{B_n} r} \, \ol{B_n} y \,dr 
= \frac{1}{h} \int_0^h T(r) B^{\circ} y \,dr \\
&\longrightarrow B^{\circ}y \quad (h \searrow 0)
\end{align*}
by the dominated convergence theorem. So, $B^{\circ}$ is closable with $\ol{B^{\circ}} \subset A_T$. 
We now want to show that $D(\ol{B^{\circ}})$ is a core for $A_T$ by verifying the invariance $T(r)D(\ol{B^{\circ}}) \subset D(\ol{B^{\circ}})$ for all $r \in [0,\infty)$. If $y \in Y^{\circ}$, then
\begin{align*}
B_m e^{ \ol{B_n} r } y = e^{ \ol{B_n} r } \big( B_m + \nu_{m,n} r  \big) y \quad \text{with} \quad 
\nu_{m,n} := \int_s^t \int_s^t \mu(r_m(\tau),r_n(\sigma)) \, d\sigma \, d\tau
\end{align*}
by the product decomposition of $e^{\ol{B_n} r }$ from~\eqref{eq: thm 2.1 zwbeh (ii), 1} and by the central commutation relation~\eqref{eq: central}. So, 
\begin{align*}
B^{\circ} e^{ \ol{B_n} r } y = e^{ \ol{B_n} r } \big( B^{\circ} + \lim_{m \to \infty} \nu_{m,n} r  \big) y
\end{align*}
for all $y \in Y^{\circ}$, from which it further follows 
that
\begin{align*}
T(r)y \in D(\ol{B^{\circ}}) \quad \text{and} \quad \ol{B^{\circ}} \, T(r)y = T(r) \big( B^{\circ} + \lim_{n \to \infty} \lim_{m \to \infty} \nu_{m,n} r  \big) y = T(r) B^{\circ}y
\end{align*}
for all $y \in Y^{\circ}$. In the last equation, we used that $\mu(\tau,\sigma) = -\mu(\sigma,\tau)$ for all $\sigma,\tau \in I$ which can be seen from~\eqref{eq: vertrel halbgr}. It follows that $\ol{B^{\circ}} T(r) \supset T(r) \ol{B^{\circ}}$ and, in particular, $T(r)D(\ol{B^{\circ}}) \subset D(\ol{B^{\circ}})$ for all $r \in [0,\infty)$. So, $D(\ol{B^{\circ}})$ is a core for $A_T$ and hence $A_T = \ol{B^{\circ}}$, as desired.
\end{proof}

We also note the following variant of the above theorem where 
the form~\eqref{eq: vertrel, p=1, direkter} of the imposed commutation relation is closer to~\eqref{eq: 1-facher kommutator skalar}. In return, one has to require relatively strong invariance conditions.

\begin{cor} \label{cor: wohlgestelltheit, p=1}
Suppose $A(t): D(A(t)) \subset X \to X$ for every $t \in I$ is the generator of a strongly continuous semigroup on $X$ such that $A$ is $(M,\omega)$-stable for some $M \in [1,\infty)$ and $\omega \in \R$. Suppose further that $Y$ is an $A(t)$-admissible subspace of $X$ for every $t \in I$ such that
\begin{align*}
Y \subset \bigcap_{\tau \in I} D(A(\tau)) \qquad \text{and} \qquad A(t)Y \subset \bigcap_{\tau \in I} D(A(\tau)),
\end{align*}
$A(t)|_Y$ is a bounded 
operator from $Y$ to $X$, and 
\begin{align} \label{eq: vertrel, p=1, direkter}
[A(s),A(t)] \big|_{D(\tilde{A}(t))} \subset \mu(s,t) \in \C
\end{align}
for all $s,t \in I$, where $\tilde{A}(t)$ is the part of $A(t)$ in $Y$.  
Suppose finally that $(s,t) \mapsto \mu(s,t)$ and $t \mapsto A(t)y$ are continuous for all $y \in Y$. 
Then the conclusions of the above theorem hold true.
\end{cor}

\begin{proof}
We verify the assumptions of the previous theorem and, to that purpose, we establish the commutation relations
\begin{align} \label{eq: vertrel halbgr, cor}
e^{A_1 \sigma} e^{A_2 \tau} = e^{A_2 \tau} e^{A_1 \sigma} e^{\mu_{1 2} \tau \sigma} \quad (\sigma, \tau \in [0,\infty)),
\end{align}
where $A_k := A(t_k)$ and $\mu_{k l} := \mu(t_k,t_l)$ for arbitrary $t_1, t_2 \in I$. In order to see~\eqref{eq: vertrel halbgr, cor}, 
one shows that 
\begin{align}  \label{eq: cor 2.2, 2}
A_1 e^{A_2 \tau} y = e^{A_2 \tau} \big( A_1 + \mu_{1 2} \tau \big)y
\end{align} 
for $y \in Y$ by differentiating $[0,\tau] \ni r \mapsto e^{A_2(\tau-r)} A_1 e^{A_2 r} y$ for vectors $y$ in the domain of the part $\tilde{A}_2$ of $A_2$ in $Y$ which by the $A_2$-admissibility of $Y$ is the generator of the strongly continuous semigroup $t \mapsto e^{A_2 t}|_Y$ in $Y$ (Proposition~2.3 of~\cite{Kato70}). (In addition to the $A_2$-admissibility, the boundedness of $A_1|_Y$ from $Y$ to $X$ and the invariance condition $A_1 Y \subset D(A_2)$ come into play here.) 
Along the same lines as~\eqref{eq: vertrel halbgr}, the relation~\eqref{eq: vertrel halbgr, cor} then follows. And since~\eqref{eq: vertrel halbgr, cor} is equivalent to the commutation relations~\eqref{eq: vertrel, p=1} and since $Y^{\circ} \supset Y$ is dense in $X$, the assumptions of Theorem~\ref{thm: wohlgestelltheit, p=1} are satisfied, as desired.
\end{proof}

\subsection{Scalar $p$-fold commutators}

In this subsection we prove well-posedness for~\eqref{eq: ivp} in the case~\eqref{eq: p-facher kommutator skalar} where the $p$-fold commutators of the operators $A(t)$ are complex scalars for some $p \in\N$. We have to make precise the merely formal commutation relation~\eqref{eq: p-facher kommutator skalar}, of course, and we begin with a well-posedness result where~\eqref{eq: p-facher kommutator skalar} is replaced by the formally equivalent commutation relations~\eqref{eq: vertrel, allg p} for the semigroups 
$e^{A(t)\,.\,}$ with the generators $A(s_1) = C^{(0)}(s_1)$ and certain operators $C^{(k)}(s_1, \dots, s_{k+1})$ which are formally given as the $k$-fold commutator of the operators $A(s_1), \dots, A(s_{k+1})$. 
%

\begin{thm} \label{thm: wohlgestelltheit, allg p}
Suppose $A(t): D(A(t)) \subset X \to X$ for every $t \in I$ is the generator of 
a strongly continuous semigroup on $X$ such that $A$ is $(M,\omega)$-stable for some $M \in [1,\infty)$ and $\omega \in \R$ and 
such that for some closed operators $C^{(k)}(s_1, \dots, s_{k+1})$, where $k \in \{0, \dots, p-1\}$ and $C^{(0)}(s) := A(s)$, and for some complex numbers $\mu(t_1, \dots, t_{p+1}) \in \C$
\begin{gather}
C^{(k)}(\ul{s}) e^{A(t) \tau} \supset e^{A(t)\tau} \big( C^{(k)}(\ul{s}) + C^{(k+1)}(\ul{s}, t) \tau + \dotsb + C^{(p-1)}(\ul{s}, t, \dots, t) \frac{\tau^{p-1-k}}{(p-1-k)!} \,+ \notag \\
\qquad \qquad \quad + \, \mu(\ul{s}, t, \dots, t) \frac{\tau^{p-k}}{(p-k)!} \big)   \qquad (\ul{s}:= (s_1, \dots, s_{k+1}))  \label{eq: vertrel, allg p}
\end{gather}
for all $k \in \{0, \dots, p-1\}$ and $s_i, t \in I$ and $\tau \in [0,\infty)$. Suppose further that the maximal continuity subspace 
\begin{gather}
Y^{\circ} := \bigcap_{k=0}^{p-1} \{ y \in D_k: (t_1, \dots, t_{k+1}) \mapsto C^{(k)}(t_1, \dots, t_{k+1}) y \text{ is continuous} \}    \label{eq: max stetur, allg p}
\\
D_k :=  \cap_{\tau_1, \dots, \tau_{k+1} \in I} D(C^{(k)}(\tau_1, \dots, \tau_{k+1})) \notag
\end{gather}
is dense in $X$ and that $(t_1, \dots, t_{p+1}) \mapsto \mu(t_1, \dots, t_{p+1})$ is continuous. Then there exists a unique evolution system $U$ for $A$ on $Y^{\circ}$.
\end{thm}

\begin{proof}
We define $U_n$ as in~\eqref{eq: def U_n, 1} and~\eqref{eq: def U_n, 2} and, for $u \in I$, we define $i_{n\, u}$ to be the index $i \in \{0, \dots, m_n\}$ with $u \in [r_{n\,i}, r_{n\,i+1})$. 
We then obtain, by the assumed commutation relations~\eqref{eq: vertrel, allg p}, the following important commutation relation which allows us to take the operators $A(r)$ from the left of $U_n(t,s)$ to the right: 
\begin{gather} \label{eq: vorbeiziehrel 2, allg p}
A(r) U_n(t,s)y = U_n(t,s) \big( A(r) + S^{(1)}_n(t,s,r) + \dotsb + S^{(p)}_n(t,s,r) \big)y \\
S^{(l)}_n(t,s,r) :=  \int_s^t \int_s^{t_n(\tau_1)} \dotsb \int_s^{ t_n(\tau_{l-1}) } C^{(l)}(r, r_n(\tau_1), 
\dots, r_n(\tau_l)) \big/ \alpha_{i_{n\,\tau_1}, \dots, i_{n\,\tau_l}} \, d\tau_l \dots d\tau_2 \, d\tau_1    \notag
\end{gather}
for all $y \in Y^{\circ}$ and $r \in I$, $(s,t) \in \Delta$ and $l \in \{1,\dots, p\}$, where $C^{(p)} := \mu$ and $t_n(\tau) := \min\{r_n^+(\tau), t\}$ 
for $\tau \in I$ and where $\alpha_{j_1, \dots, j_l}$ for an $l$-tupel $(j_1, \dots, j_l)$ of natural numbers denotes the number of permutations $\sigma$ leaving the $l$-tupel invariant, that is, 
\begin{align*}
(j_{\sigma(1)}, \dots, j_{\sigma(l)}) = (j_1, \dots, j_l).
\end{align*} 
(In verifying~\eqref{eq: vorbeiziehrel 2, allg p}, it is best to write 
$A(r) = A_r = C^{(0)}_r$ and $U_n(t,s) = e^{A_m h_m} \dotsb e^{A_1 h_1}$  with $A_j = A(s_j)$ and to prove by induction over $m \in \N$, 
with the help of the assumed commutation relations, that
\begin{gather*}
A_r e^{A_m h_m} \dots e^{A_1 h_1} y = e^{A_m h_m} \dots e^{A_1 h_1} \big( A_r + S^{(1)} + \dotsb + S^{(p)} \big) y \\
S^{(l)} := \sum_{1 \le j_l \le \dotsb \le j_1 \le m} C^{(l)}_{r; j_1, \dots, j_l} \big / \alpha_{j_1, \dots, j_l} \,\, h_{j_1} \dotsb h_{j_l}
\quad \text{with} \quad C^{(l)}_{r; j_1, \dots, j_l} := C^{(l)}(r,s_{j_1}, \dots, s_{j_l}).
\end{gather*}
It is easy to see that the sums $S^{(l)}$ are nothing but the integrals $S^{(l)}_n(t,s,r)$ in~\eqref{eq: vorbeiziehrel 2, allg p} and therefore~\eqref{eq: vorbeiziehrel 2, allg p} follows.)
With the help of the commutation relation~\eqref{eq: vorbeiziehrel 2, allg p}, the continuity of the maps $(t_1, \dots, t_{k+1}) \mapsto C^{(k)}(t_1, \dots, t_{k+1})y$ for $y \in Y^{\circ}$, the fact that $\alpha_{i_{n\,\tau_1}, \dots, i_{n\,\tau_k}} \longrightarrow k!$ as $n \to \infty$ for every $(\tau_1, \dots, \tau_k) \in I^k$ with $\tau_1 > \dotsb > \tau_k$, and the closedness of the operators $A(r)$, we see in the same way as in the proof of Theorem~\ref{thm: wohlgestelltheit, p=1} that
\begin{itemize}
\item $(U_n(t,s)x)$ is a Cauchy sequence in $X$ uniformly in $(s,t) \in \Delta$ for every $x \in X$ with limit denoted by $U(t,s)x$,
\item $U(t,s)U(s,r) = U(t,r)$ for every $(r,s), (s,t) \in \Delta$ and $(s,t) \mapsto U(t,s)$ is strongly continuous,
\item $[s,1] \ni t \mapsto U(t,s)y$ for every $y \in Y^{\circ}$ is a continuously differentiable solution to~\eqref{eq: ivp}.
\end{itemize}
Consequently, $U$ is at least an evolution system for $A$ on $Y^{\circ}$ in the wide sense, and it remains to show that $[s,1] \ni t \mapsto U(t,s)y$ has values in $Y^{\circ}$ for every $y \in Y^{\circ}$. In order to do so one establishes, using the same arguments as for~\eqref{eq: vorbeiziehrel 2, allg p}, the commutation relation
\begin{align}  \label{eq: vorbeiziehrel 3, allg p}
C^{(k)}(\ul{r}) U_n(t,s)y = U_n(t,s) \big( C^{(k)}(\ul{r}) + S^{(k+1)}_n(t,s,\ul{r}) + \dotsb + S^{(p)}_n(t,s,\ul{r}) \big)y
\end{align}
for all $y \in Y^{\circ}$ and $\ul{r} \in I^{k+1}$, $(s,t) \in \Delta$ and $k \in \{0, \dots, p-1\}$, where $S_n^{(k+l)}(t,s,\ul{r})$ is defined as the integral of 
\begin{align*}
(\tau_1, \dots, \tau_l) \mapsto C^{(k+l)}(\ul{r}, r_n(\tau_1), \dots, r_n(\tau_l)) \big/ \alpha_{i_{n\,\tau_1}, \dots, i_{n\,\tau_l}}
\end{align*}
over the same domain of integration as in the definition of $S^{(l)}_n(t,s,r)$ in~\eqref{eq: vorbeiziehrel 2, allg p}. 
Since the operators $C^{(k)}(\ul{r})$ are closed for $\ul{r} \in I^{k+1}$ and $k \in \{0, \dots, p-1\}$ by assumption, 
it follows from~\eqref{eq: vorbeiziehrel 3, allg p} that for every $y \in Y^{\circ}$ and $k \in \{0, \dots, p-1\}$ one has: 
\begin{align*}
U(t,s)y \in D(C^{(k)}(\ul{r})) \text{\, for every } \ul{r} \in I^{k+1} \text{\, and \,} \ul{r} \mapsto C^{(k)}(\ul{r}) U(t,s)y \text{\, is continuous}
\end{align*}
or, in other words, that $U(t,s)y \in Y^{\circ}$ for every $y \in Y^{\circ}$, as desired.
\end{proof}

We also note the following variant of the above theorem where the form~\eqref{eq: vertrel, allg p, direkter} of the imposed commutation relation is closer to~\eqref{eq: p-facher kommutator skalar}. In return, one has to require relatively strong invariance conditions.

\begin{prop} \label{prop: wohlgestelltheit, allg p}
Suppose $A(t): D(A(t)) \subset X \to X$ for every $t \in I$ is the generator of a strongly continuous semigroup on $X$ such that $A$ is $(M,\omega)$-stable for some $M \in [1,\infty)$ and $\omega \in \R$ and recursively define $C^{(0)}(t) := A(t)$ as well as $C^{(k)}(t_1, \dots, t_{k+1}) := [C^{(k-1)}(t_1, \dots, t_k), A(t_{k+1})]$ for $k \in \N$. Suppose further that $Y$ is an $A(t)$-admissible subspace of $X$ for every $t \in I$, and $p \in \N$ a natural number such that for all $t_i \in I$
\begin{align*}
Y \subset \bigcap_{\tau_1, \dots, \tau_p \in I} D(C^{(p-1)}(\tau_1, \dots, \tau_p)) \qquad \text{and} \qquad C^{(p-1)}(t_1, \dots, t_p) Y \subset \bigcap_{\tau \in I} D(C^{(0)}(\tau)),
\end{align*}
$C^{(k)}(t_1, \dots, t_{k+1})|_Y$ is a bounded 
operator from $Y$ to $X$ for all $k \in \{ 0, \dots, p-1\}$, and 
\begin{align} \label{eq: vertrel, allg p, direkter}
C^{(p)}(t_1, \dots, t_{p+1}) \big|_{D(\tilde{A}(t_{p+1}))} \subset \mu(t_1, \dots, t_{p+1}) \in \C,
\end{align}
where $\tilde{A}(t)$ is the part of $A(t)$ in $Y$.
Suppose finally that $(t_1, \dots, t_{p+1}) \mapsto \mu(t_1, \dots, t_{p+1})$ and $(t_1, \dots, t_{k+1}) \mapsto C^{(k)}(t_1, \dots, t_{k+1})y$ are continuous for all $y \in Y$ and $k \in \{0, \dots, p-1\}$. Then there exists a unique evolution system $U$ in the wide sense for $A$ on $Y$. 
\end{prop}

\begin{proof}
We recall that, by our convention from the beginning of Section~2, the commutators $C^{(k)}(t_1, \dots, t_{k+1})$ are to be understood in the operator-theoretic sense, and we can therefore conclude that 
\begin{align*}
Y \subset \bigcap_{\tau_1, \dots, \tau_{k+1} \in I} D(C^{(k)}(\tau_1, \dots, \tau_{k+1})) \qquad \text{and} \qquad C^{(k)}(t_1, \dots, t_{k+1}) Y \subset \bigcap_{\tau \in I} D(C^{(0)}(\tau)),
\end{align*}
for all $k \in \{0, \dots, p-1\}$ by successively proceeding from $p-1$ to $0$. 
With this in mind, one verifies the commutation relations
\begin{gather}
C^{(k)}(\ul{s}) e^{A(t) \tau}y = e^{A(t)\tau} \big( C^{(k)}(\ul{s}) + C^{(k+1)}(\ul{s}, t) \tau + \dotsb + C^{(p-1)}(\ul{s}, t, \dots, t) \frac{\tau^{p-1-k}}{(p-1-k)!} \, + \notag \\
+ \, \mu(\ul{s}, t, \dots, t) \frac{\tau^{p-k}}{(p-k)!} \big)y \qquad (\ul{s}:= (s_1, \dots, s_{k+1}))  \label{eq: vertrel, allg p, unhandl}
\end{gather}
for all $y \in Y$ and $k \in \{0, \dots, p-1\}$ by proceeding from $p-1$ to $0$ and by using, at each successive step, the same arguments as for~\eqref{eq: cor 2.2, 2}. And from~\eqref{eq: vertrel, allg p, unhandl}, in turn, one obtains the existence of an evolution system $U$ in the wide sense for $A$ on $Y$ in exactly the same way as in the proof of Theorem~\ref{thm: wohlgestelltheit, allg p}. (It is not to be expected, however, that $U$ is even an evolution system for $A$ on $Y$ in the strict sense. See the sixth remark in Section~2.4.) In order to obtain uniqueness, one has only to observe that for any evolution system $V$ in the wide sense for $A$ on $Y$,
\begin{align*}
U_n(t,s)y - V(t,s)y = V(t,\tau) U_n(\tau,s)y \big|_{\tau=s}^{\tau=t} = \int_s^t V(t,\tau)\big( A(r_n(\tau))-A(\tau)\big) U_n(\tau,s)y \, d\tau 
\end{align*} 
converges to $0$ for every $y \in Y$ and $(s,t) \in \Delta$ by 
\eqref{eq: vorbeiziehrel 2, allg p}. 
\end{proof}

\subsection{Well-posedness for group generators}

After having proved well-posedness results for semigroup generators with~\eqref{eq: 1-facher kommutator skalar} or~\eqref{eq: p-facher kommutator skalar}, we now improve, inspired by~\cite{Kobayasi79}, the special well-posedness result from~\cite{Kato70} (Theorem~5.2 in conjunction with Remark~5.3) for a certain kind of group (instead of semigroup) generators $A(t)$ and certain uniformly convex subspaces $Y$ of the domains $D(A(t))$: we show that this result is still valid if $t \mapsto A(t)|_Y$ is assumed to be only strongly continuous (instead of norm continuous as in~\cite{Kato70}). In~\cite{Kobayasi79} the same is done for the general well-posedness theorem from~\cite{Kato70} (Theorem~6.1). We point out that although several arguments from~\cite{Kobayasi79} can be used here as well, it is by no means obvious that the improvement made in~\cite{Kobayasi79} can be carried over to the special well-posedness result of~\cite{Kato70}. In particular, the possibility of such an improvement is not mentioned in the literature -- at least, not in~\cite{Kobayasi79}, \cite{Yagi79}, \cite{Yagi80}, \cite{Kato85}, \cite{Kato93}, \cite{Tanaka00}, \cite{Tanaka01}.

\begin{thm} \label{thm: wohlg, grerz}
Suppose $A(t): D(A(t)) \subset X \to X$ for every $t \in I$ is the generator of a strongly continuous group on $X$ such that $A^+ := A(\,.\,)$ and $A^- := -A(1-\,.\,)$ are $(M,\omega)$-stable for some $M \in [1,\infty)$ and $\omega \in \R$. Suppose further that $Y$ for every $t \in I$ is an $A^{\pm}(t)$-admissible subspace of $X$ contained in $\cap_{\tau \in I} D(A(\tau))$ 
and that $A(t)|_Y$ is a bounded operator from $Y$ to $X$ such that 
\begin{align*}
t \mapsto A(t)|_Y
\end{align*} 
is strongly continuous. And finally, suppose there exists for each $t \in I$ a norm $\norm{\,.\,}_t^{\pm}$ on~$Y$ equivalent to the original norm of $Y$ such that $Y_t^{\pm} := (Y, \norm{\,.\,}_t^{\pm})$ is uniformly convex and
\begin{align} \label{eq: normen komp}
\norm{y}_t^{\pm} \le e^{c^{\pm} |t-s|} \norm{y}_s^{\pm} \quad (y \in Y \text{ and } s,t \in I)
\end{align}
for some constant $c^{\pm} \in (0,\infty)$ and such that the $Y$-part $\tilde{A}^{\pm}(t)$ of $A^{\pm}(t)$ generates a quasicontraction semigroup in $Y_t^{\pm}$, more precisely
\begin{align} \label{eq: qkontrgr in Y_t}
\norm{e^{\tilde{A}^{\pm}(t) \tau}y}_t^{\pm} 
\le  e^{\omega_0 \tau} \norm{y}_t^{\pm} \quad (\tau \in [0,\infty), y \in Y, t \in I)
\end{align}
for some $t$-independent growth exponent 
$\omega_0 \in \R$. Then there exists a unique evolution system $U$ for $A$ on $Y$.
\end{thm}

\begin{proof}
We adopt from~\cite{Kobayasi79} the shorthand notation $U^{\pm}(t,s,\pi)$ for products of the semigroups $e^{A^{\pm}(t)\,.\,}$ associated with finite or infinite partitions $\pi$ in $I$. Without further specification, convergence or continuity in $X$, $Y$ will always mean convergence or continuity in 
the norm of $X$, $Y$.
\smallskip

As a first step we show that for each $y \in Y$ and $s \in [0,1)$ there exists a sequence $(\pi_n^{\pm}) = (\pi_{y,s,n}^{\pm})$ of 
partitions of $I$ such that $(U^{\pm}(t,s,\pi_{y,s,n}^{\pm})y)$ is a Cauchy sequence in $X$ for $t \in [s,1]$. What we have to show here is that for every sequence $\pi = (t_k)$, strictly monotonically increasing in $I$, and arbitrary $t_k' \in [t_k, t_{k+1})$, the following assertions are satisfied (Lemma~1 of~\cite{Kobayasi79}): 
\begin{itemize}
\item[(i)] $(U^{\pm}(t_k',t_0,\pi)x)$ is a Cauchy sequence in $X$ for every $x \in X$ whose limit will be denoted by $U^{\pm}(t_{\infty},t_0,\pi)x$ where $t_{\infty} := \lim_{k \to \infty} t_k'$, 
\item[(ii)] $(U^{\pm}(t_k',t_0,\pi)y)$ is a Cauchy sequence in $Y$ for every $y \in Y$.   
\end{itemize}
With the help of Lemma~2 and~3 of~\cite{Kobayasi79}, whose proofs carry over without change to the present situation, the existence of sequences $(\pi_{y,s,n}^{\pm})$ of partitions with the claimed properties then follows. 
Assertion~(i) is simple and is proven in the same way as in~\cite{Kobayasi79}, while assertion~(ii) has to be proven in a completely different way because the proof of~\cite{Kobayasi79} essentially rests on the existence of certain isomorphisms $S(t)$ from $Y$ onto $X$ which are not available here. 
We show, using ideas from~\cite{Kato70} (Section~5), that 
\begin{align} \label{eq: U(t_k',t_0,pi)y schwach konv}
U^{\pm}(t_{\infty},t_0,\pi)y \in Y \quad \text{and} \quad U^{\pm}(t_k',t_0,\pi)y \longrightarrow U^{\pm}(t_{\infty},t_0,\pi)y 
\quad \text{weakly in } Y
\end{align} 
for every $y \in Y$ and that
\begin{align}  \label{eq: limsup U(t_k',t_0,pi)y}
\limsup_{k \to \infty} \norm{ U^{\pm}(t_k',t_0, \pi)y }_{\ol{t}_{\infty}}^{\mp} \le \norm{ U^{\pm}(t_{\infty},t_0, \pi)y }_{\ol{t}_{\infty}}^{\mp} \quad (\ol{t}_{\infty} := 1-t_{\infty})
\end{align}
for $y \in Y$, which two things by the uniform convexity of $Y_{\ol{t}_{\infty}}$ imply the convergence 
\begin{align*}
U^{\pm}(t_k',t_0,\pi)y \longrightarrow U^{\pm}(t_{\infty},t_0,\pi)y  \text{\, in } Y 
\end{align*}
and in particular assertion~(ii).
In order to see~\eqref{eq: U(t_k',t_0,pi)y schwach konv} notice first that $\tilde{A}^{\pm}$ 
is $(\tilde{M}, \tilde{\omega})$-stable for some $\tilde{M} \in [1,\infty)$ and $\tilde{\omega} \in \R$ by~\eqref{eq: normen komp} and~\eqref{eq: qkontrgr in Y_t}  (Proposition~3.4 of~\cite{Kato70}), so that the sequence $(U^{\pm}(t_k',t_0,\pi)y)$ is bounded in the norm of $Y$ (recall that
\begin{align*}
e^{A^{\pm}(t) \tau} \big|_Y = e^{ \tilde{A}^{\pm}(t) \tau}
\end{align*}
by Proposition~2.3 of~\cite{Kato70}). Since $Y$ is reflexive (Milman's theorem), every subsequence of $(U^{\pm}(t_k',t_0,\pi)y)$ has in turn a weakly convergent subsequence in $Y$ whose weak limit must be equal to $U^{\pm}(t_{\infty},t_0,\pi)y$ by assertion~(i), and therefore~\eqref{eq: U(t_k',t_0,pi)y schwach konv} follows.
In order to see~\eqref{eq: limsup U(t_k',t_0,pi)y} notice first that $(U^{\pm}(t_k',t_n, \pi)x)_{n \in \N}$ is a Cauchy sequence in $X$ for every $x \in X$ and $k \in \N$, where
\begin{align*}
U^{\pm}(t_k',\tau, \pi) := U^{\pm}(\tau, t_k', \pi)^{-1} 
= e^{-A^{\pm}(t_k)(t_{k+1}-t_k')} \dotsb e^{-A^{\pm}(r_{\pi}(\tau))(\tau-r_{\pi}(\tau))} 
\end{align*} 
for $\tau \in (t_k',t_{\infty})$ 
and where $r_{\pi}(\tau)$ denotes the largest point of $\pi$ less than or equal to $\tau$. Indeed, for every $x \in Y$, 
\begin{align*}
U^{\pm}(t_k',t_m,\pi)x - U^{\pm}(t_k',t_n,\pi)x = -\int_{t_m}^{t_n} U^{\pm}(t_k',\tau,\pi) A^{\pm}(r_{\pi}(\tau))x \, d\tau \longrightarrow 0 \quad (m,n \to \infty) 
\end{align*} 
in $X$ and by the $(M,\omega)$-stability of $A^{\mp}$, this convergence extends to all $x \in X$. We denote the limit by $U^{\pm}(t_k',t_{\infty},\pi)x$ and note for later use that 
\begin{align} \label{eq: schritt 1, schw konv}
U^{\pm}(t_k', t_{\infty},\pi)y \in Y \quad \text{and} \quad U^{\pm}(t_k',t_n,\pi)y \longrightarrow U^{\pm}(t_k', t_{\infty},\pi)y \quad \text{weakly in } Y
\end{align}
by the same arguments as those for~\eqref{eq: U(t_k',t_0,pi)y schwach konv}.
Since $U^{\pm}(t_k',t_0,\pi) = U^{\pm}(t_k',t_n,\pi) U^{\pm}(t_n,t_0,\pi)$ for all $n \in \N$, it follows that 
\begin{align}  \label{eq: schritt 1, flusseigenschaft}
U^{\pm}(t_k',t_0,\pi) = U^{\pm}(t_k',t_{\infty},\pi) U^{\pm}(t_{\infty},t_0,\pi).
\end{align}
Also, since 
\begin{align*}
U^{\pm}(t_k',t_n,\pi) = e^{A^{\mp}( \ol{t}_k )(t_{k+1}-t_k')} \dotsb e^{A^{\mp}( \ol{t}_{n-1} )(t_n-t_{n-1})} \quad (\ol{t}_i := 1 - t_i)
\end{align*}
for $n \ge k+1$, it follows by successively passing from $\norm{\,.\,}_{\ol{t}_{\infty}}^{\mp}$ to $\norm{\,.\,}_{\ol{t}_k}^{\mp}$ to ...~to $\norm{\,.\,}_{\ol{t}_{n-1}}^{\mp}$ and back to $\norm{\,.\,}_{\ol{t}_{\infty}}^{\mp}$ with the help of~\eqref{eq: normen komp}, and by using~\eqref{eq: qkontrgr in Y_t} at each successive step, that 
\begin{align*}
\norm{ U^{\pm}(t_k',t_n,\pi)z }_{\ol{t}_{\infty}}^{\mp} \le e^{2c^{\mp} (t_{\infty}-t_k)} \, e^{\omega_0 (t_n-t_k')} \norm{z}_{\ol{t}_{\infty}}^{\mp}
\end{align*}
for every $z \in Y$, and therefore
\begin{align}  \label{eq: schritt 1, limsup}
\norm{ U^{\pm}(t_k',t_{\infty}, \pi)z }_{\ol{t}_{\infty}}^{\mp} 
\le e^{2c^{\mp} (t_{\infty}-t_k)} \, e^{\omega_0 (t_{\infty}-t_k')} \norm{z}_{\ol{t}_{\infty}}^{\mp}
\end{align}
for $z \in Y$ by virtue of~\eqref{eq: schritt 1, schw konv}. 
Combining now~\eqref{eq: schritt 1, flusseigenschaft} and~\eqref{eq: schritt 1, limsup} we obtain~\eqref{eq: limsup U(t_k',t_0,pi)y}, which concludes our first step. 
\smallskip

As a second step we observe that $U_0^{\pm}(t,s)y := \lim_{n \to \infty} U^{\pm}(t,s,\pi_{y,s,n}^{\pm})y$ for $y \in Y$ and $(s,t) \in \Delta$ defines a 
linear operator from $Y$ to $X$ extendable to a bounded operator $U^{\pm}(t,s)$ in $X$, and that $U^{\pm}$ is an evolution system in $X$ such that $t \mapsto U^{\pm}(t,s)y$ for every $y \in Y$ is right differentiable (in the norm of $X$) at $s$ with right derivative $A^{\pm}(s)y$. 
All this follows in the same way as in~\cite{Kobayasi79} (Lemma~4 and~5). In particular, it follows from the right differentiability and evolution system properties just mentioned that $[0,t] \ni s \mapsto U^{\pm}(t,s)y$ is continuously differentiable (from both sides) for every $y \in Y$ with derivative $s \mapsto -U^{\pm}(t,s)A^{\pm}(s)y$ by Corollary~2.1.2 of~\cite{Pazy83}.
\smallskip

As a third step we show that $U^{\pm}(t,s)$ leaves the subspace $Y$ invariant for every $(s,t) \in \Delta$ and that $[s,1] \ni t \mapsto U^{\pm}(t,s)y$ is right continuous in $Y$ for every $y \in Y$. 
In order to see that $U^{\pm}(t,s)y$ lies in $Y$ for $y \in Y$, notice that the sequence $(U^{\pm}(t,s,\pi_{y,s,n}^{\pm})y)$ is bounded in the norm of $Y$, 
whence by the same argument as for~\eqref{eq: U(t_k',t_0,pi)y schwach konv} 
\begin{align} \label{eq: schritt 3, schw konv}
U^{\pm}(t,s)y \in Y \quad 
\text{and} \quad U^{\pm}(t,s,\pi_{y,s,n}^{\pm})y \longrightarrow U^{\pm}(t,s)y 
\quad \text{weakly in } Y.
\end{align} 
In order to see that $[s,1] \ni t \mapsto U^{\pm}(t,s)y$ is right continuous in $Y$ for every $y \in Y$, we have only to show, by the invariance property just established, that $U^{\pm}(t+h,t)y \longrightarrow y$ in $Y$ as $h \searrow 0$ for every $t \in [0,1)$. And for this in turn it is sufficient to show, by the uniform convexity of $Y_t$, that 
\begin{gather}
U^{\pm}(t+h,t)y \longrightarrow y \quad \text{weakly in } Y \text{ as } h \searrow 0 \\
\text{and} \notag \\
\limsup_{h \searrow 0} \norm{ U^{\pm}(t+h,t)y }_t^{\pm} \le \norm{y}_t^{\pm}
\end{gather}
Since this can be achieved in a way similar to the proof of~\eqref{eq: U(t_k',t_0,pi)y schwach konv} and~\eqref{eq: limsup U(t_k',t_0,pi)y}, 
we may omit the details.
\smallskip

We can now 
show that $t \mapsto U^+(t,s)y$ is continuous in $Y$ for every $y \in Y$ and then conclude the proof. 
Indeed, $\tau \mapsto U^{\mp}(1-s,1-\tau)z$ is differentiable for $z \in Y$ with derivative $\tau \mapsto U^{\mp}(1-s,1-\tau)A^{\mp}(1-\tau)z$ by the last remark of our second step 
and $\tau \mapsto U^{\pm}(\tau,s)y$ is right differentiable for $y \in Y$ with right derivative $\tau \mapsto A^{\pm}(\tau)U^{\pm}(\tau,s)y$ because for every $\tau \in [s,1)$ the vector $z :=  U^{\pm}(\tau,s)y$ lies in $Y$ and
\begin{align*}
\frac 1 h \big( U^{\pm}(\tau+h,s)y - U^{\pm}(\tau,s)y \big) = \frac 1 h \big( U^{\pm}(\tau+h,\tau)z - z \big) \longrightarrow A^{\pm}(\tau)z \quad (h \searrow 0) 
\end{align*}
by our second and third step. So, the map $[s,t] \ni \tau \mapsto U^{\mp}(1-s,1-\tau)U^{\pm}(\tau,s)y$ is right differentiable for every $y \in Y$ with right derivative $0$. Corollary~2.1.2 of~\cite{Pazy83} therefore yields
\begin{align*}
U^{\mp}(1-s,1-t)U^{\pm}(t,s)y - y = U^{\mp}(1-s,1-\tau)U^{\pm}(\tau,s)y \big|_{\tau=s}^{\tau=t} = 0
\end{align*}
for every $y \in Y$ and hence
\begin{align*}
U^{\mp}(1-s,1-t)U^{\pm}(t,s) = 1 = U^{\mp}(1-\ol{t},1-\ol{s})U^{\pm}(\ol{s},\ol{t}) = U^{\mp}(t,s)U^{\pm}(1-s,1-t)
\end{align*} 
for all $(s,t) \in \Delta$. It follows that 
\begin{align*}
U^+(t-h,s)y = U^+(t,t-h)^{-1} U^+(t,s)y = U^-(1-t+h,1-t) U^+(t,s)y \longrightarrow U^+(t,s)y 
\end{align*}
in $Y$ as $h \searrow 0$ by our third step, whence $t \mapsto U^+(t,s)y$ right \emph{and} left continuous and hence continuous in $Y$. 
Combining this with the previous steps, we see with the help of Corollary~2.1.2 of~\cite{Pazy83} that $t \mapsto U^+(t,s)y$ is continuously differentiable in $X$ for every $y \in Y$ with derivative $t \mapsto A^+(t)U^+(t,s)y$ and therefore $U := U^+$ is an evolution system for $A = A^+$ on $Y$, as desired.
\end{proof}

Incidentally, it is also possible to improve (a version of) the well-posedness theorem from~\cite{Kato73} (Theorem~1) in the spirit of~\cite{Kobayasi79}: in this theorem strong continuity of $t \mapsto A(t)|_Y$ is sufficient as well, provided that $A$ is $(M,\omega)$-stable (instead of only quasistable) and that $t \mapsto \norm{B(t)}$ is bounded (instead of only upper integrable). 
(We make this proviso in order to make sure that the boundedness condition~(2.1) of~\cite{Kobayasi79} is still satisfied for arbitrary 
partitions $\pi$ and that~(2.2) of~\cite{Kobayasi79} is satisfied with the modified right hand side $C \norm{x} \int_{t_i}^{t_k'} \alpha(\tau) \, d\tau$, where $\alpha$ is a suitable integrable function. All other arguments from~\cite{Kobayasi79} carry over without formal change, a bit more care being necessary in the justification of assertion~(c) of~\cite{Kobayasi79} because of the weaker regularity of $t \mapsto S(t)$ -- see~\cite{Dorroh75}.)

\subsection{Some remarks}

We close this section about abstract well-posedness results with some remarks concerning, in particular, the relation of the results from Section~2.1 and~2.2 with the results from~\cite{Kato70}, \cite{Kato73}, \cite{Kobayasi79}, \cite{NickelSchnaubelt98} and the result from Section~2.3. 
\smallskip

1. Compared to the well-posedness theorems from~\cite{Kato70}, \cite{Kato73}, \cite{Kobayasi79} where no commutator conditions of the kind~\eqref{eq: 1-facher kommutator skalar} or~\eqref{eq: p-facher kommutator skalar} are imposed, the well-posedness theorems from Section~2.1 and~2.2 are furnished with rather mild stability and regularity conditions: 
%
%
Concerning stability, we had only to require in the theorems from Section~2.1 and~2.2 that the family
$A$ be $(M,\omega)$-stable in $X$ (or that the slightly weaker stability condition~\eqref{eq: stab nickelschnaubelt} be satisfied). In the well-posedness theorems from~\cite{Kato70}, 
\cite{Kato73}, \cite{Kobayasi79}, by contrast, it has to be required in addition that there exist an $A(t)$-admissible subspace $Y$ of $X$ contained in all the domains of the $A(t)$ such that the induced family $\tilde{A}$ consisting of 
the $Y$-parts $\tilde{A}(t)$ of the $A(t)$ 
is $(\tilde{M},\tilde{\omega})$-stable in $Y$. Such a subspace $Y$ is generally difficult to find 
-- unless the domains of the $A(t)$ are time-independent. (In this latter case, one can choose $Y := D(A(0)) = D(A(t))$ endowed with the graph norm of $A(0)$, provided only that $t \mapsto A(t)$ 
is of bounded variation -- just apply 
Proposition~4.4 of~\cite{Kato70} with $S(t) := A(t)-(\omega+1)$.)
%
%
Concerning regularity, we had only to require strong continuity conditions in the theorems from Section~2.1 and~2.2: namely, we had to require that 
\begin{align*}
t \mapsto C^{(0)}(t)y = A(t)y \quad \text{and} \quad  (t_1,\dots,t_{k+1}) \mapsto C^{(k)}(t_1, \dots, t_{k+1})y 
\end{align*}
be continuous for $k \in \{1, \dots, p\}$ and $y$ in a dense subspace $Y$ of $X$ contained in all the respective domains or, equivalently, that the maximal continuity subspaces~\eqref{eq: max stetur, p=1} or~\eqref{eq: max stetur, allg p} be dense in $X$ and that $\mu$ be continuous. 
In general situations without commutator conditions of the kind~\eqref{eq: 1-facher kommutator skalar} or~\eqref{eq: p-facher kommutator skalar}, by contrast, strong continuity conditions are not sufficient for well-posedness -- not even if the domains of the $A(t)$ are time-independent. (See the respective counterexamples in~\cite{Phillips53} (Example~6.4), \cite{EngelNagel00} (Example~VI.9.21), \cite{SchmidGriesemer15} (Example~1 and~2).)
Accordingly, in the general well-posedness results from~\cite{Kato70} (Theorem~6.1), \cite{Kobayasi79}, and~\cite{Kato73} (Theorem~1) for general semigroup generators $A(t)$, there is a strong $W^{1,1}$-regularity condition on certain auxiliary operators $S(t)$ defined on an $A(t)$-admissible subspace $Y$ of $X$ contained in all the domains $D(A(t))$, 
which boils down 
to a strong $W^{1,1}$-regularity condition on $t \mapsto A(t)$ in the case of time-independent domains $D(A(t)) = Y$ (Remark~6.2 of~\cite{Kato70}); 
and in the special well-posedness result (Theorem~5.2 and Remark~5.3) from~\cite{Kato70} for group generators $A(t)$, there still is a norm continuity condition on $t \mapsto A(t)|_Y$ and a regularity condition on certain auxiliary norms $\norm{\,.\,}_t^{\pm}$ on $Y$, which boils down to a Lipschitz continuity condition on $t \mapsto A(t)$ in the case of time-independent domains $D(A(t)) = Y$ (Theorem~2.1 of~\cite{SchmidGriesemer15}).
\smallskip

2. In a certain special case involving group generators $A(t)$ with time-independent domains, the well-posedness assertion of the theorems from Section~2.1 and~2.2 can alternatively also be inferred from 
the well-posedness theorem from Section~2.3. In fact, if in addition to the assumptions of Theorem~2.3 the following three conditions are satisfied, then the well-posedness assertion of this theorem (but no representation formula, of course) also follows from Theorem~\ref{thm: wohlg, grerz}:
\begin{itemize}
\item $A(t)$ for every $t \in I$ is a quasicontraction group generator with time-independent domain $D(A(t)) = Y$ in the uniformly convex space $X$ such that 
\begin{align} \label{eq: zshg mit spez wohlgsatz, 0}
\norm{e^{\pm A(t) \tau} } \le e^{\omega \tau} \quad (\tau \in [0,\infty))
\end{align}
for some $t$-independent growth exponent $\omega \in \R$,
\item $C^{(k)}(t_1, \dots, t_{k+1})$ is a bounded operator on $X$ for every $(t_1, \dots, t_{k+1}) \in I^{k+1}$ and 
\begin{align} \label{eq: zshg mit spez wohlgsatz, 1}
\sup_{ (t_1, \dots, t_{k+1}) \in I^{k+1} } \norm{ C^{(k)}(t_1, \dots, t_{k+1}) } < \infty
\end{align}
for every $k \in \{1, \dots, p-1\}$ (an empty condition for $p=1$!),
\item $t \mapsto A(t)y$ is continuous for every $y \in Y$. 
\end{itemize}
Indeed, under these conditions the norms $\norm{\,.\,}_t^{\pm}$ appearing in Theorem~\ref{thm: wohlg, grerz} can be chosen to be $\norm{\,.\,}_* := \norm{(A(0)-\omega-1)\,.\,}$ for every $t \in I$ ($t$-independent!): with this norm, $Y$~becomes a uniformly convex subspace admissible for the group generators $\pm A(t)$ and
\begin{align} \label{eq: zshg mit spez wohlgsatz, 2}
\norm{e^{\pm A(t)\tau}y}_* \le e^{\omega_0 \tau} \norm{y}_* \quad (y \in Y \text{ and } \tau \in [0,\infty))
\end{align}
for a suitable $\omega_0 \in \R$, and finally $Y^{\circ} = Y$. 
(In order to see~\eqref{eq: zshg mit spez wohlgsatz, 2} and the $\pm A(t)$-admissibility of $Y$ one checks that 
\eqref{eq: vertrel, allg p} holds true for $\tau \in (-\infty,0)$ as well, so that in particular
\begin{align*}
A(0) e^{\pm A(t)\tau}y = e^{\pm A(t)\tau} \big( A(0) &+ C^{(1)}(0,t) (\pm \tau) + \dotsb + C^{(p-1)}(0,t, \dots, t) (\pm \tau)^{p-1}/(p-1)! \\
&+ \mu(0,t, \dots, t) (\pm \tau)^p / p! \big)y
\end{align*}
for all $y \in Y$ and $\tau \in [0,\infty)$. With the help of~\eqref{eq: zshg mit spez wohlgsatz, 0} and \eqref{eq: zshg mit spez wohlgsatz, 1} the desired $\pm A(t)$-admissibility 
and the quasicontraction group property~\eqref{eq: zshg mit spez wohlgsatz, 2} then readily follow.)
\smallskip

3. In the well-posedness theorems from~\cite{Ishii82} and~\cite{NeidhardtZagrebnov09} weaker notions of well-posedness are used than here~\cite{NeidhardtZagrebnov14}, which in return allows for weaker regularity assumptions than those of~\cite{Kato73} and~\cite{Kobayasi79} (but 
the stability conditions are the same). 
In the second product representation theorem from~\cite{Nickel00} (Proposition~4.9) which also asserts well-posedness, there seems to be missing, in the hyperbolic case, an additional stability and regularity assumption of the kind of condition~(ii'') from~\cite{Kato70}. At least, it is not clear~\cite{NagelNickelSchnaubelt14} how the asserted well-posedness should be established and how the range condition from Chernoff's theorem (invoked in~\cite{Nickel00}) should be verified without such an additional assumption. (In this respect, see in particular Theorem~4.19 of~\cite{Nickel96} and the remarks preceding it, which state that $\mathcal{Y}$ is a core for $\mathcal{G}$ only under the additional condition (ii'') from~\cite{Kato70}.) 
As far as~\cite{Constantin01} is concerned, it should be remarked that the abstract well-posedness theorem of this paper is actually a corollary of the well-posedness theorem of~\cite{Kato73}. (Indeed, if for every $y \in Y$ the map $t \mapsto S(t)y$ is differentiable at all except countably many points with an exceptional set $N$ not depending on $y$ and if $\sup_{t \in I \setminus N} \norm{S'(t)y} < \infty$, then $t \mapsto S(t)y$ is already absolutely continuous (Theorem~6.3.11 of~\cite{Cohn13}) and
\begin{align*}
S(t)y = S(0)y + \int_0^t S'(\tau)y \, d\tau
\end{align*}
(Proposition~1.2.3 of~\cite{ArendtBatty12}) for every $y \in Y$, so that the strong $W^{1,1}$-regularity condition for $t \mapsto S(t)$ from~\cite{Kato73} is satisfied.) 
\smallskip

4. It is clear from the proofs of Theorem~\ref{thm: wohlgestelltheit, p=1} and Theorem~\ref{thm: wohlgestelltheit, allg p} that the well-posedness statements remain valid if the $(M,\omega)$-stability condition of these theorems is replaced by the condition from~\cite{NickelSchnaubelt98} that there exist a sequence $(\pi_n)$ of partitions of $I$ such that $\operatorname{mesh}(\pi_n) \longrightarrow 0$ and
\begin{align} \label{eq: stab nickelschnaubelt}
\norm{ e^{A(r_n(t)) (t-r_n(t))} 
\dotsb e^{A(r_n(s)) (r_n^+(s)-s)} } \le M e^{\omega (t-s)} \quad ((s,t) \in \Delta).
\end{align} 
In~\cite{NickelSchnaubelt98} this stability condition is shown to be strictly weaker than $(M,\omega)$-stability.
%
Also, it is clear from the proof of Theorem~\ref{thm: wohlgestelltheit, p=1} that the representation formula for the evolution is still valid if~\eqref{eq: stab nickelschnaubelt} is sharpened to
\begin{align} \label{eq: stab nickelschnaubelt, strikt}
\norm{ e^{A(r_n(t)) r (t-r_n(t))} \dotsb e^{A(r_n(s)) r (r_n^+(s)-s)} } \le M e^{\omega r (t-s)} \quad ((s,t) \in \Delta, r \in [0,\infty)).
\end{align} 
In particular, the method of proof of Theorem~\ref{thm: wohlgestelltheit, p=1} yields an alternative and more elementary proof (without reference to the Trotter--Kato theorem) of Proposition~2.5 from~\cite{NickelSchnaubelt98} (or, rather, of a slightly corrected version of it: in order for the proof of~\cite{NickelSchnaubelt98} to work one has to choose as the domain of $\int_s^t A(\tau) \,d\tau$ the maximal continuity subspace $Y^{\circ}$ of $A$ as defined in~\eqref{eq: max stetur, p=1}, instead of the quite arbitrary subspace denoted by $Y$ in~\cite{NickelSchnaubelt98} because such a subspace, in contrast to~$Y^{\circ}$, is not left invariant by $(\ol{B}_n-\lambda)^{-1}$ in general). 
\smallskip

5. In the situation of Theorem~2.1, 
one might think that it should be possible to (more efficiently) obtain the well-posedness of the initial value problems~\eqref{eq: ivp} on $Y^{\circ}$ by first defining a candidate $U$ for the sought evolution system through the representation formula
\begin{align*}
U(t,s) := e^{ \ol{ ( \int_s^t A(\tau) \, d\tau )^{\circ} } } e^{1/2 \int_s^t \int_s^{\tau} \mu(\tau, \sigma) \, d\sigma \, d\tau }, 
\end{align*}
and by then verifying that this candidate is indeed an evolution system for $A$ on $Y^{\circ}$. 
In order to prove that the closure of $(\int_s^t A(\tau)\, d\tau)^{\circ}$ exists and is a semigroup generator, 
one might want to employ the theorem of Trotter and Kato as in~\cite{NickelSchnaubelt98} -- instead of exploiting the locally uniform convergence of the sequences $(U_n^r(t,s)x)$ as we did. 
And in order to verify the evolution system properties for 
$U$, one might want to make rigorous the following formal differentiation rule for exponential operators (appearing in~\cite{Wilcox67}, for instance): 
\begin{align} \label{eq: duhamel}
\frac{e^{B(t+h)} - e^{B(t)}}{h} &= \frac{e^{B(t+h) \tau} e^{B(t)(1-\tau)}}{h} \Big|_{\tau=0}^{\tau=1}  
= \int_0^1 e^{B(t+h) \tau} \frac{B(t+h)-B(t)}{h} \, e^{B(t)(1-\tau)} \, d\tau \notag \\
&\longrightarrow \int_0^1 e^{B(t+h) \tau} B'(t) \, e^{B(t)(1-\tau)} \, d\tau \quad (h \to 0)
\end{align} 
with $B(t) := \ol{ ( \int_s^t A(\tau) \, d\tau )^{\circ} }$. 
Yet, this is possible only if $\Re \mu(\tau,\sigma) \ge 0$ for all $\sigma \le \tau$ because only then can the right hand side of~\eqref{eq: thm 2.1 zwbeh (ii), 1} be dominated by a bound $M' e^{\omega' r}$ for all $r \in [0,\infty)$ uniformly in $n \in \N$ (a first crucial assumption of the Trotter--Kato theorem). And moreover, the verification of the density of $\ran{ ( ( \int_s^t A(\tau) \, d\tau )^{\circ} - \lambda) }$ in $X$ for $\lambda > \omega'$ (a second crucial assumption of the Trotter--Kato theorem) and the verifications of the evolution system properties for $U$ with the help of~\eqref{eq: duhamel} are more involved than the arguments in our approach.
\smallskip

6. In 
Proposition~2.4 we obtained well-posedness on the given subspace $Y$ only in the wide sense, that is, the existence of a unique evolution system $U$ for $A$ on $Y$ in the wide sense. We could not prove the invariance of $Y$ under the operators $U(t,s)$, however (while in the special case $p=1$ we could prove such an invariance for a different subspace, namely~\eqref{eq: max stetur, p=1}, in Corollary~2.2). 
And, in fact, we do not expect it to be true in general: 
at least, it is 
not possible to obtain this invariance -- as in Theorem~2.3 -- 
by a closedness argument from~\eqref{eq: vorbeiziehrel 3, allg p} (which equation 
is still true in the situation of Proposition~2.4 for vectors $y \in Y$) 
because in general, under the assumptions of Proposition~2.4, none of the operators $C^{(k)}(t_1, \dots, t_{k+1})|_Y$ will be a closed operator in $X$.
Choose, for instance, 
\begin{align*}
A(t) := A_0 + B(t) = A_0 + b(t) B_0 \text{ \, in } X := L^1(I), 
\end{align*}
where $A_0 f := \partial_x f$ for  $f \in D(A_0) = \{ f \in W^{1,1}(I): f(1) = 0\}$ 
and $(B_0 f)(x) := x^p f(x)$ for $f \in X$ and where $t \mapsto b(t) \in \C$ is continuous, and then choose
\begin{align*}
Y := D(A_0^2) \text{\, if } p = 1 \quad \text{and} \quad Y := D(A_0^p) \text{\, if } p \in \N \setminus \{1\}
\end{align*}
endowed with the norm of $W^{2,1}(I)$ or $W^{p,1}(I)$, respectively. It is then easy to see that, indeed, all the assumptions of Proposition~2.4 are satisfied, but $C^{(k)}(t_1, \dots, t_{k+1})|_Y$ is non-closed for every $k \in \{0, \dots, p\}$ and every $(t_1, \dots, t_{k+1}) \in I^{k+1}$. 
(We point out that in this specific example one nevertheless does have the invariance of $Y$ under the operators $U(t,s)$ for $s \le t$, but this seems to essentially depend on 
the specific structure of the example: since $A(t)$ is $A_0$ plus a bounded perturbation $B(t)$, the evolution system $U$ for $A$ on $Y$ in the wide sense satisfies
\begin{align*}
U(t,s)f = e^{A_0(t-s)}f + \int_s^t e^{A_0(t-\tau)} B(\tau) U(\tau,s)f \, d\tau \quad (f \in X)
\end{align*}
and hence is given by the 
respective perturbation series expansion; and since $e^{A_0\,.\,}|_Y$ is a strongly continuous semigroup in $Y$ and $B(t)|_Y$ is a bounded operator in $Y$, 
this perturbation series leaves $Y$ invariant.) 
%

\section{Some applications of the abstract results}

We now discuss some applications of the abstract results from Section~2. In all of them the operators $A(t)$ will be skew self-adjoint in a Hilbert space $X$.

\subsection{Segal field operators}

In this subsection we apply the well-posedness result of Section~2.1 to Segal field operators $\Phi(f_t)$ in $\mathcal{F}_+(\mathfrak{h})$, the symmetric Fock space over a complex Hilbert space $\mathfrak{h}$. 
Segal field operators $\Phi(f)$ are defined for $f \in \mathfrak{h}$ as the closure of $2^{-1/2}(a(f) + a^*(f))$, where $a(f)$ and $a^*(f)$ are the usual annihilation and creation operators in $\mathcal{F}_+(\mathfrak{h})$ corresponding to $f$. It is well-known that the operators $\Phi(f)$ are self-adjoint and, as a consequence of the canonical commutation relations for creation and annihilation operators, they satisfy the commutation relations 
\begin{align} \label{eq: ccr, 0}
[\Phi(f), \Phi(g)] = i \Im \scprd{f,g} \quad (f,g \in \mathfrak{h})
\end{align}
on a suitable dense subspace of $\mathcal{F}_+(\mathfrak{h})$. 
See~\cite{BratteliRobinson} (Section~5.2.1), \cite{ReedSimon} (Section~X.7) or~\cite{Dimock11} (Section~5.4) for these and other standard facts about such operators and basic concepts from quantum field theory. 
In view of~\eqref{eq: ccr, 0} we expect to obtain well-posedness for the operators $A(t) = i \Phi(f_t)$ by means of Theorem~\ref{thm: wohlgestelltheit, p=1}. Indeed, we have:


\begin{cor} \label{cor: wohlg segal}
Set $A(t) = i \Phi(f_t)$ in $X := \mathcal{F}_+(\mathfrak{h})$ and suppose that $t \mapsto f_t \in \mathfrak{h}$ is continuous. Then there exists a unique evolution system $U$ for $A$ on the maximal continuity subspace $Y^{\circ}$ for $A$ 
and it is given by
\begin{align*}
U(t,s) = e^{ \ol{ (\int_s^t i \Phi(f_{\tau}) \, d\tau)^{\circ} } } e^{-i/2 \int_s^t \int_s^{\tau} \Im \scprd{f_{\tau}, f_{\sigma}} \, d\sigma \, d\tau}
= W\Big(\int_s^t f_{\tau} \, d\tau \Big) e^{-i/2 \int_s^t \int_s^{\tau} \Im \scprd{f_{\tau}, f_{\sigma}} \, d\sigma \, d\tau} 
\end{align*}
where $W(h) := e^{i \Phi(h)}$ denotes the Weyl operator for $h \in \mathfrak{h}$. 
\end{cor}

\begin{proof}
We have already remarked that the operators $A(t)$ are skew self-adjoint and hence (semi)group generators. We also see, by the Weyl form 
\begin{align} \label{eq: ccr}
\Phi(f) e^{i \Phi(g)} = e^{i \Phi(g)} \big( \Phi(f) - \Im \scprd{f,g} \big) 
\end{align}
of the canonical commutation relations (Proposition~5.2.4~(1) in~\cite{BratteliRobinson}), that the generators $A(s)$ can be commuted through the groups $e^{A(t)\,.\,}$ 
in the way required in~\eqref{eq: vertrel, p=1} with $\mu(s,t) = -i \Im \scprd{f_s,f_t}$. 
It remains to show that the maximal continuity subspace $Y^{\circ}$ for $A$ is a dense subspace of $X$. In order to do so, one uses that for every $f \in \mathfrak{h}$ one has: $D(N^{1/2}) \subset D(\Phi(f))$ 
and 
\begin{align} \label{eq: Phi rel beschr bzgl N^1/2}
\norm{\Phi(f)\psi} = \big\| 2^{-1/2}(a(f) + a^*(f)) \psi \big\|  \le 2^{1/2} \norm{f} \big\| (N+1)^{1/2} \psi \big\|
\end{align}
for every $\psi \in D(N^{1/2})$ (Lemma~5.3 of~\cite{Dimock11}), where $N$ is the number operator in $\mathcal{F}_+(\mathfrak{h})$. Since $t \mapsto f_t$ is continuous by assumption, the estimate~\eqref{eq: Phi rel beschr bzgl N^1/2} shows that the maximal continuity subspace $Y^{\circ}$ for $A$ contains the dense subspace $D(N^{1/2})$ of $X$ and is therefore dense itself. So, the desired well-posedness statement and the first of the 
asserted representation formulas for $U$ follow from Theorem~\ref{thm: wohlgestelltheit, p=1}. 
In order to see the second representation formula for $U$, repeatedly apply the identity 
\begin{align}
W(f)W(g) = W(f+g)e^{-i/2 \Im \scprd{f,g}}
\end{align}
(Proposition~5.2.4~(2) of~\cite{BratteliRobinson}) to the approximants $U_n$ for $U$ from the proof of Theorem~\ref{thm: wohlgestelltheit, p=1} and use the strong continuity of $\mathfrak{h} \ni h \mapsto W(h)$ (Proposition~5.2.4~(4) of~\cite{BratteliRobinson}). 
Alternatively, the well-posedness statement and the first representation formula could also be concluded from Corollary~\ref{cor: wohlgestelltheit, p=1} with $Y := D(N)$ endowed with the graph norm of $N$. 
Indeed, $Y$ with this norm is an $A(t)$-admissible subspace of $X$ because 
\begin{align*}
N e^{i \Phi(f)} = e^{i \Phi(f)}\big( N + \Phi(if) + \norm{f}^2 /2 \big)
\end{align*}
for all $f \in \mathfrak{h}$ (Proposition~2.2 of~\cite{Merkli06}), $Y \subset \cap_{\tau \in I} D(A(\tau))$ and $A(t)Y \subset D(N^{1/2}) \subset \cap_{\tau \in I} D(A(\tau))$
by the definition of creation and annihilation operators, $A(t)|_Y$ is a bounded operator from $Y$ to $X$ by~\eqref{eq: Phi rel beschr bzgl N^1/2}, and finally $[A(s), A(t)]|_{D(N)} \subset -i \Im \scprd{f_s,f_t}$ (Proposition~5.2.3~(3) of~\cite{BratteliRobinson}).
\end{proof}

It is possible to give at least two 
alternative proofs of variants of the above result and we briefly comment on these alternative approaches (which, however, are not necessary for understanding Corollary~3.2 below). 
A first alternative approach is based upon the fifth remark from Section~2.4, which is applicable here because $\Re \mu(\tau,\sigma) = 0$ for all $\sigma, \tau \in I$. It yields the following version of Corollary~3.1: if $t \mapsto f_t \in \mathfrak{h}$ is continuous, then there exists a unique evolution system $U$ for $A$ on the maximal continuity subspace $Y^{\circ}$ for $A$ and $U$ is given by the first representation formula of the corollary. 
A second alternative -- and more pedestrian -- approach is based upon a well-known exponential series expansion for Weyl operators, namely~\eqref{eq: weylop reihendarst} below, and yields the following version of Corollary~3.1 for $\mathfrak{h} = L^2(\R^3)$: if both  
\begin{align*}
t \mapsto f_t \in \mathfrak{h} = L^2(\R^3) \quad \text{and} \quad t \mapsto f_t/\sqrt{\omega} \in \mathfrak{h} = L^2(\R^3)
\end{align*}
are continuous for a measurable function $\omega: \R^3 \to \R$ with $\omega(k) > 0$ for almost all $k \in \R^3$, then there exists a unique evolution system $U$ in the wide sense for $A$ on $Y = D(H_{\omega}^{1/2})$ (where $H_{\omega}$ is the second quantization of $\omega$ defined in~\eqref{eq: def H_omega, 1} and~\eqref{eq: def H_omega, 2} below) and $U$ is given by the second representation formula from the corollary above.
In order to see this by pedestrian arguments, one defines a candidate for the sought evolution $U$ in the wide sense through
\begin{align}  \label{eq: def U, pedestrian}
U(t,s) := W\Big(\int_s^t f_{\tau} \, d\tau \Big) \,\, e^{-i/2 \int_s^t \int_s^{\tau} \Im \scprd{f_{\tau}, f_{\sigma}} \, d\sigma \, d\tau}
\end{align}
and exploits the exponential series expansion 
\begin{align} \label{eq: weylop reihendarst} 
W(g) \psi = e^{i \Phi(g)} \psi = \sum_{n=0}^{\infty} \frac{i^n}{n!} \, \Phi(g)^n \psi \qquad (g \in \mathfrak{h})
\end{align}
for Weyl operators $W(g)$ on vectors $\psi$ in the finite particle subspace $\mathcal{F}_{+}^0(\mathfrak{h}) := \{ \psi \in \mathcal{F}_+(\mathfrak{h}): \psi^{(n)} = 0 \text{ for all but finitely many } n \}$. (See the proof of Theorem~X.41 of~\cite{ReedSimon}.)
With this expansion, one can show by term-wise differentiation and repeated application of the commutation relation $[\Phi(f), \Phi(g)]|_{\mathcal{F}_{+}^0(\mathfrak{h})} \subset i \Im \scprd{f,g}$ (Theorem~X.41~(c) of~\cite{ReedSimon}) 
that the mapping $t \mapsto U(t,s)\psi$
is differentiable for $\psi \in \mathcal{F}_{+}^0(\mathfrak{h})$ with the desired derivative $t \mapsto i \Phi(f_t) U(t,s)\psi$. 
Since $\mathcal{F}_{+}^0(\mathfrak{h})$ is a core for $\Phi(f_t)|_Y$ uniformly in $t \in I$ by virtue of~\eqref{eq: absch Phi durch H_w^{1/2}} below 
(recall, $Y = {D(H_{\omega}^{1/2})}$) and since the operators $\Phi(f_t)$ can be commuted through $U(t,s)$ up to scalar errors by virtue of~\eqref{eq: ccr},
there exists for every $\psi \in Y$ 
a sequence $(\psi_n)$ in $\mathcal{F}_{+}^0(\mathfrak{h})$ such that $\psi_n \longrightarrow \psi$ and 
\begin{align*}
i \Phi(f_t) U(t,s) \psi_n \longrightarrow U(t,s) \Big( i \Phi(f_t)\psi - i \int_s^t \Im \scprd{ f_t, f_\tau} \,d\tau \psi \Big) = i \Phi(f_t) U(t,s)\psi
\end{align*}
uniformly in $t \in I$ as $n \to \infty$. It follows that $t \mapsto U(t,s)\psi = \lim_{n \to \infty} U(t,s)\psi_n$ is continuously differentiable even for $\psi \in Y$ with the desired derivative. So, $U$ defined by~\eqref{eq: def U, pedestrian} is indeed an evolution system in the wide sense for $A$ on $Y = D(H_{\omega}^{1/2})$, and it is also unique by virtue of~\cite{Kato53} (Theorem~1).
\smallskip

With the help of the above well-posedness result for Segal field operators we will now establish 
the well-posedness of the initial value problems for operators $H_{\omega} + \Phi(f_t)$ in $\mathcal{F}_+(\mathfrak{h})$ with $\mathfrak{h} := L^2(\R^3)$. 
Such operators are sometimes called van Hove Hamiltonians 
and they describe a classical particle coupled to a time-dependent quantized field of bosons: 
$H_{\omega}$ describes the energy of the field while $\Phi(f_t)$ describes the interaction of the particle with the field. (See, for instance,~\cite{Derezinski03} or~\cite{KelerTeufel12}.) The operator $H_{\omega}$ is the second quantization of the dispersion relation $\omega: \R^3 \to \R$, a measurable function with $\omega(k) > 0$ for almost every $k \in \R^3$, 
that is, $H_{\omega}$ is the operator on $\mathcal{F}_+(\mathfrak{h}) = \bigoplus_{n \in \N \cup \{0\}} \mathfrak{h}^{(n)}_+$ defined by 
\begin{align} \label{eq: def H_omega, 1}
(H_{\omega} \psi)^{(n)} := H_{\omega}^{(n)} \psi^{(n)}  \text{\,\, for \,} \psi \in D(H_{\omega}) := \{ \psi 
\in \mathcal{F}_+(\mathfrak{h}): (H_{\omega}^{(n)} \psi^{(n)}) \in \mathcal{F}_+(\mathfrak{h}) \},
\end{align}
where the operators $H_{\omega}^{(n)}$ act by multiplication as follows: 
\begin{align} \label{eq: def H_omega, 2}
H_{\omega}^{(0)} \psi^{(0)}:= 0 \quad \text{and} \quad (H_{\omega}^{(n)} \psi^{(n)})(k_1, \dots, k_n) := \sum_{i=1}^n \omega(k_i) \psi^{(n)}(k_1, \dots, k_n)
\end{align}
for $\psi^{(0)} \in \mathfrak{h}^{(0)}_+ = \C$ and $\psi^{(n)} \in  \mathfrak{h}^{(n)}_+ = L^2_+(\R^{3n}) := \{ \varphi \in L^2(\R^{3n}): \varphi(k_{\sigma(1)}, \dots, k_{\sigma(n)}) = \varphi(k_1, \dots, k_n) \text{ for all permutations } \sigma \}$ (Example~1 in Section~X.7 of~\cite{ReedSimon}).
It is well-known and easy to see that $H_{\omega}$ is a positive self-adjoint operator and that for all $f$ with $f \in \mathfrak{h}$ and $f/\sqrt{\omega} \in \mathfrak{h}$ one has: $D(H_{\omega}^{1/2}) \subset D(\Phi(f))$ and
\begin{align} \label{eq: absch Phi durch H_w^{1/2}}
\norm{\Phi(f)\psi} \le  2^{1/2} \big( \norm{f}^2 + \norm{f/\sqrt{\omega}}^2 \big)^{1/2} \| (H_{\omega}+1)^{1/2} \psi \|
\end{align}
for all $\psi \in D(H_{\omega}^{1/2})$. 
(See, for instance, (13.70) of~\cite{Spohn04} or (20.33) and (20.34) of~\cite{GustafsonSigal11}.) With the help of~\eqref{eq: absch Phi durch H_w^{1/2}} it easily follows that $\Phi(f)$ is infinitesimally bounded w.r.t.~$H_{\omega}$ and hence 
that $H_{\omega} + \Phi(f)$ is self-adjoint on $D(H_{\omega})$ provided $f \in \mathfrak{h}$ and $f/\sqrt{\omega} \in \mathfrak{h}$.

\begin{cor} \label{cor: wohlg qft}
Set $A(t) = -i (H_{\omega} + \Phi(f_t))$ in $X := \mathcal{F}_+(\mathfrak{h})$, where $\mathfrak{h} := L^2(\R^3)$ and $\omega: \R^3 \to \R$ is measurable with $\omega(k) > 0$ for almost all $k \in \R^3$, and suppose that 
$t \mapsto f_t/\sqrt{\omega} \in \mathfrak{h}$ is continuous and $t \mapsto f_t \in \mathfrak{h}$ is absolutely continuous. 
Then there exists a unique evolution system $U$ for $A$ on $D(H_{\omega})$ and it is given by~\eqref{eq: darst qft, 1} and~\eqref{eq: darst qft, 2} below.
\end{cor}

\begin{proof}
It follows from the remarks above that the operators $A(t)$ are skew self-adjoint with time-independent domain $D(H_{\omega})$ because $f_t, f_t/\sqrt{\omega} \in \mathfrak{h}$ by assumption. Since at least formally $[H_{\omega}, i \Phi(g)] = \Phi(i \omega g)$ by virtue of Lemma~2.5~(ii) of~\cite{DerezinskiGerard99}, 
the $p$-fold commutators~\eqref{eq: p-facher kommutator skalar} will not collapse to a complex scalar in general. We can therefore not hope to apply the results from Section~2.1 and~2.2 directly. We can, however, reduce the desired assertion to Corollary~3.1 by switching to the interaction picture, that is, we define 
a candidate for the sought evolution system $U$
as the interaction picture evolution, 
\begin{align} \label{eq: darst qft, 1}
U(t,s) := e^{-i H_{\omega}t} \tilde{U}(t,s) e^{i H_{\omega}s},
\end{align}
where $\tilde{U}$ denotes the evolution system for $\tilde{A}$ with $\tilde{A}(t) := -i e^{i H_{\omega} t} \Phi(f_t) e^{-i H_{\omega} t}$. 
It has to be shown, of course, that this evolution exists on an appropriate dense subspace, and this can be done by way of Corollary~3.1. Indeed, 
\begin{align} \label{eq: Phi(f) an e^(i H_w)}
e^{i H_{\omega} t} \Phi(f) e^{-i H_{\omega} t} = \Phi(e^{i \omega t} f) \quad (f \in \mathfrak{h}, t \in \R)
\end{align}
by Theorem~X.41~(e) of~\cite{ReedSimon}, that is, the operator $\tilde{A}(t)$ is ($i$ times) a Segal field operator,
\begin{align*}
\tilde{A}(t) = -i e^{i H_{\omega} t} \Phi(f_t) e^{-i H_{\omega} t} = i \Phi(\tilde{f}_t) \quad \text{with} \quad \tilde{f}_t := -e^{i \omega t} f_t
\end{align*}
and $t \mapsto \tilde{f}_t$ is obviously continuous. So, by Corollary~\ref{cor: wohlg segal}, the evolution system $\tilde{U}$ exists on the maximal continuity subspace $\tilde{Y}^{\circ}$ for $\tilde{A}$ and is given by
\begin{align} \label{eq: darst qft, 2}
\tilde{U}(t,s) = W(g_{t,s}) e^{-i/2 \int_s^t \int_s^{\tau} \Im \scprd{\tilde{f}_{\tau}, \tilde{f}_{\sigma}} \, d\sigma \, d\tau} \quad \text{with \,\,} g_{t,s} := \int_s^t \tilde{f}_{\tau} \, d\tau. 
\end{align}
We now show that $U$, given by~\eqref{eq: darst qft, 1} and~\eqref{eq: darst qft, 2}, is an evolution system for $A$ on $D(H_{\omega})$. 
In order to see that $t \mapsto U(t,s)\psi$ is differentiable for all $\psi \in D(H_{\omega})$ with the desired derivative 
\begin{align} \label{eq: qft, formel für abl}
t \mapsto -i(H_{\omega} + \Phi(f_t)) U(t,s)\psi,
\end{align} 
we have to 
show in view of~\eqref{eq: darst qft, 1} that 
\begin{align} \label{eq: qft, zwbeh 0}
D(H_{\omega}) \subset \tilde{Y}^{\circ} \quad \text{and} \quad \tilde{U}(t,s) D(H_{\omega}) \subset  D(H_{\omega}) \quad ((s,t) \in \Delta).
\end{align}
And in order to show that~\eqref{eq: qft, formel für abl} 
is continuous, we would like to move the unbounded operators $H_{\omega}$ and $\Phi(f_t)$ through the constituents $e^{-i H_{\omega}t}$ and $W(g_{t,s})$ of $U(t,s)$ in a suitable way. 
We first show the inclusion~(\ref{eq: qft, zwbeh 0}.a) 
and the continuity of $t \mapsto \Phi(f_t)U(t,s)\psi$ for all $\psi \in D(H_{\omega})$. 
It easily follows from~\eqref{eq: absch Phi durch H_w^{1/2}} and the assumed continuity of $t \mapsto f_t, f_t/\sqrt{\omega} \in \mathfrak{h}$ that
\begin{align} \label{eq: D(H_w) in Y}
D(H_{\omega}) \subset D(H_{\omega}^{1/2}) \subset \tilde{Y}^{\circ}
\end{align}
and hence that~(\ref{eq: qft, zwbeh 0}.a) holds true.
It also follows from~\eqref{eq: Phi(f) an e^(i H_w)} and~\eqref{eq: ccr} that
\begin{gather*}
\Phi(f_t) e^{-i H_{\omega}t} = e^{-i H_{\omega}t} \Phi(e^{i \omega t} f_t), \\
\Phi(e^{i \omega t} f_t) W(g_{t,s}) = W(g_{t,s}) \big( \Phi(e^{i \omega t} f_t) - \Im \scprd{ e^{i \omega t} f_t, g_{t,s} } \big).
\end{gather*}
So, $t \mapsto \Phi(f_t) e^{-i H_{\omega}t} W(g_{t,s}) \psi$ and hence $t \mapsto \Phi(f_t) U(t,s)\psi$ is continuous for $\psi \in D(H_{\omega})$ because $t \mapsto \Phi(e^{i \omega t} f_t) \psi$ is continuous for $\psi \in D(H_{\omega})$ by~\eqref{eq: absch Phi durch H_w^{1/2}} and because $t \mapsto W(g_{t,s})$ is strongly continuous by Proposition~5.2.4~(4) of~\cite{BratteliRobinson}.
We now show the inclusion~(\ref{eq: qft, zwbeh 0}.b) and the continuity of $t \mapsto H_{\omega}U(t,s)\psi$ for all $\psi \in D(H_{\omega})$ by showing that $W(g_{t,s}) D(H_{\omega}) \subset D(H_{\omega})$ and that $H_{\omega}$ can be moved through $W(g_{t,s})$ in a suitable way. It is here that the assumed absolute continuity of $t \mapsto f_t$ will come into play.
Since 
\begin{align} \label{eq: H_w and W(g) vorbei}
W(g)D(H_{\omega}) = D(H_{\omega}) \quad \text{and} \quad H_{\omega} W(g) = W(g)\big( H_{\omega} + \Phi(i \omega g) + \scprd{g, \omega g}/2 \big)
\end{align}
for every $g \in D(\omega) = \{ h \in \mathfrak{h}: \omega h \in \mathfrak{h} \}$ (Lemma~2.5~(ii) of~\cite{DerezinskiGerard99}), 
we are led to showing that
\begin{align} \label{eq: qft, zwbeh}
g_{t,s}  \in D(\omega) \quad \text{and} \quad t \mapsto \omega g_{t,s} \in \mathfrak{h} \text{ is continuous.}
\end{align}
In order to do so, notice that the map $\tau \mapsto f_{\tau}$, being absolutely continuous with values in the reflexive space $\mathfrak{h}$, is differentiable almost everywhere (Corollary~1.2.7 of~\cite{ArendtBatty12}) and that $\tau \mapsto e^{i \omega \tau} (i\omega + 1)^{-1}$ is strongly continuously differentiable. We can therefore (Proposition~1.2.3 of~\cite{ArendtBatty12}) 
perform the following integration by parts:
\begin{align*}
-g_{t,s} 
&= \int_s^t e^{i \omega \tau} (i\omega + 1)^{-1} f_{\tau} \,d\tau + \int_s^t e^{i \omega \tau} i \omega \,  (i\omega + 1)^{-1} f_{\tau} \,d\tau \\
&= (i\omega + 1)^{-1} \Big( \int_s^t e^{i \omega \tau}  f_{\tau} \,d\tau + e^{i \omega \tau} f_{\tau} \Big|_{\tau=s}^{\tau=t} - \int_s^t e^{i \omega \tau} f_{\tau}' \,d\tau \Big).
\end{align*}
So, \eqref{eq: qft, zwbeh} follows. 
With the help of~\eqref{eq: qft, zwbeh} we now obtain from~\eqref{eq: H_w and W(g) vorbei} the following conclusions: 
first, that $W(g_{t,s})D(H_{\omega}) = D(H_{\omega})$ for all $(s,t) \in \Delta$ and hence that~(\ref{eq: qft, zwbeh 0}.b) holds true and second, that $t \mapsto H_{\omega}W(g_{t,s})\psi$ and hence $t \mapsto H_{\omega}U(t,s)\psi$ is continuous for $\psi \in D(H_{\omega})$ because $t \mapsto \Phi(i \omega g_{t,s}) \psi$ is continuous for $\psi \in D(H_{\omega})$ by~\eqref{eq: absch Phi durch H_w^{1/2}} and~\eqref{eq: qft, zwbeh} and because $t \mapsto W(g_{t,s})$ is strongly continuous by Proposition~5.2.4~(4) of~\cite{BratteliRobinson}.
\end{proof}

If one suitably strengthens or modifies the assumptions of Corollary~3.2, one can conclude the well-posedness statement of that corollary (but not the representation~\eqref{eq: darst qft, 1} and~\eqref{eq: darst qft, 2} for the evolution, of course) by means of various general well-posedness theorems. 
Indeed, if for instance one adds the assumption that $t \mapsto f_t/\sqrt{\omega} \in \mathfrak{h}$ be absolutely continuous as well, 
then the well-posedness statement of Corollary~3.2 can also be concluded from~\cite{Kato73} (Theorem~1) 
because, under the thus strengthened assumptions, the strong $W^{1,1}$-regularity condition on $t \mapsto A(t)$ required in~\cite{Kato73} 
can be 
verified 
by means of~\eqref{eq: absch Phi durch H_w^{1/2}}.
Similarly, if one replaces the absolute continuity condition on $t \mapsto f_t$ by the assumption that both $t \mapsto f_t$ and $t \mapsto f_t/\sqrt{\omega}$ be of bounded variation and continuous, then the well-posedness statement of Corollary~3.2 can be concluded from~\cite{Kato53} (Theorem~3). It is not difficult to find functions $\omega$ and $f_t$ as in the above corollary such that $t \mapsto f_t /\sqrt{\omega}$ is not of bounded variation (so that Corollary~\ref{cor: wohlg qft} does not follow from~\cite{Kato53}). 
Choose, for instance, $f_0 \in \mathfrak{h}$ with $f_0(k) = 1$ for $|k| \le 1$, and $\alpha \in [3/2,3)$ and then set
\begin{align*}
\omega(k) := |k|^{\alpha} \quad \text{and} \quad f_t(k) := e^{i \omega(k)^{-1/2} t} f_0(k)  \quad (k \in \R^3).
\end{align*}

\subsection{Schrödinger operators for external electric fields}

In this subsection we apply the well-posedness result of Section~2.2 to Schrödinger operators $-\Delta + b(t) \cdot x$ in $L^2(\R^d)$. Such operators describe a quantum particle in a time-dependent spatially constant electric field $b(t) \in \R^d$ and they are shown to be essentially self-adjoint below. Setting $A(t) = \ol{i \Delta - i b(t) \cdot x}$, we obtain by formal computation 
\begin{gather} \label{eq: vertrel, schrödinger, formal}
[A(t_1), A(t_2)] = 2 \sum_{\kappa=1}^d (b_{\kappa}(t_2)-b_{\kappa}(t_1)) \partial_{\kappa}, \quad
\big[ [A(t_1), A(t_2)], A(t_3) \big] = \mu(t_1,t_2,t_3) 
\end{gather}
with $\mu(t_1,t_2,t_3) := -2i \sum_{\kappa=1}^d (b_{\kappa}(t_2)-b_{\kappa}(t_1)) b_{\kappa}(t_3) \in \C$. 
In view of~\eqref{eq: vertrel, schrödinger, formal} we expect to obtain well-posedness for the operators $A(t)$ by means of Theorem~\ref{thm: wohlgestelltheit, allg p} with $p=2$. Indeed, we have (see also the remarks below):

\begin{cor} \label{cor: schrödinger}
Set $A(t) = \ol{A_0 + B(t)}$ in $X := L^2(\R^d)$ (existence of the closure is shown below), 
where $A_0 := i \Delta$ with $D(A_0) = W^{2,2}(\R^d)$ and where $B(t)$ is multiplication by $-i b(t) \cdot x$, and suppose $t \mapsto b(t) \in \R^d$ is continuous. Then there exists a unique evolution system $U$ for $A$ on the maximal continuity subspace $Y^{\circ}$ for $A = C^{(0)}$ and $C^{(1)}$ defined in~\eqref{eq: def C^(1), schrödinger}.  
Additionally, $U$ is given by~\eqref{eq: darstformel schrödinger, 1} and~\eqref{eq: darstformel schrödinger, 2} below.
\end{cor}

\begin{proof}
(i) We first show that $A_0 + B(t_0)$ for every $t_0 \in I$ is essentially skew self-adjoint and that the unitary group generated by $A := \ol{A_0+B(t_0)}$ is given by
\begin{align} \label{eq: darst T}
e^{A t} = e^{A_0 t} e^{B t} e^{-\partial_1 b_1 t^2} \dotsb e^{-\partial_d b_d t^2} e^{2i b^2 t^3/3} \quad (t \in \R),
\end{align}  
where $B := B(t_0)$ and $b = (b_1, \dots, b_d) := b(t_0) \in \R^d$. We do so by showing that the right hand side of~\eqref{eq: darst T}, which we abbreviate as 
$T(t)$, defines a strongly continuous unitary group in $X$ with
\begin{align*}
A_0 + B \subset A_T \quad \text{and} \quad T(t) D(A_0+B) \subset D(A_0+B) \quad (t \in \R),
\end{align*}
where $A_T$ stands for the generator of $T$. (In order to understand why $e^{A \,.\,}$ should decompose as in~\eqref{eq: darst T}, plug the following formal commutators
\begin{align*}
[B,A_0] = - 2 \sum_{\kappa=1}^d b_{\kappa} \partial_{\kappa}, \quad [[B,A_0], B] = 2i b^2, \quad  [[B,A_0], A_0] = 0
\end{align*}
into the Zassenhaus formula~\cite{Magnus54}, \cite{Suzuki77}, \cite{CasasMurua12} for bounded operators.)
With the help of the explicit formulas for the groups $e^{A_0 \,.\,}$ (free Schrödinger group), $e^{B \,.\,}$ (multiplication group), $e^{\partial_{\kappa} \,.\,}$ (translation group) we find the following commutation relations,
\begin{gather} 
e^{A_0 t} e^{\partial_{\kappa}s} = e^{\partial_{\kappa}s} e^{A_0 t}, \quad  e^{B t} e^{\partial_{\kappa}s} = e^{\partial_{\kappa}s} e^{B t} e^{i b_{\kappa} ts}, \notag \\
e^{A_0 t} e^{B s} = e^{B s} e^{A_0 t} e^{2 \partial_1 b_1 ts} \dotsb e^{2 \partial_d b_d ts} e^{-ib^2 t s^2} \qquad (s,t \in \R). \label{eq: schrödinger, vertrel}
\end{gather}
It follows from~\eqref{eq: schrödinger, vertrel} that $T$ is indeed a strongly continuous unitary group and that
\begin{gather*}
e^{\partial_{\kappa} s} D(A_0) \subset D(A_0), \quad e^{\partial_{\kappa} s} D(B) \subset D(B), \quad
e^{B s} D(A_0) \subset D(A_0), \\
 e^{A_0 t} D(A_0+B) \subset D(B) \qquad (s,t \in \R),
\end{gather*}
so that 
$T(t) D(A_0+B) \subset D(A_0+B)$ for all $t \in \R$ and $A_0 + B \subset A_T$. Consequently, $A_0 + B$ is essentially skew self-adjoint and $A = \ol{A_0 + B}$ is equal to $A_T$. 
After these preparations we can now verify the assumptions of Theorem~\ref{thm: wohlgestelltheit, allg p} for $p = 2$. Indeed, using the commutation relations~\eqref{eq: schrödinger, vertrel} we find that
\begin{align} \label{eq: schrödinger, vertrel halbgr}
e^{C_{1 2} \sigma} e^{A_3 \tau} = e^{A_3 \tau} e^{C_{1 2} \sigma} e^{\mu_{1 2 3} \tau \sigma}, \quad
e^{A_1 \sigma} e^{A_2 \tau} = e^{A_2 \tau} e^{A_1 \sigma} e^{C_{1 2} \tau \sigma} e^{\mu_{1 2 2} \tau^2 \sigma/2} e^{\mu_{1 2 1} \tau \sigma^2/2} 
\end{align}
for all $\sigma, \tau \in \R$, where $A_j := A(t_j) = C^{(0)}(t_j)$, $b_j := b(t_j)$, 
$\mu_{j k l} := -2 i \sum_{\kappa=1}^d (b_{k \, \kappa}-b_{j \, \kappa}) b_{l \, \kappa}$, and
\begin{align} \label{eq: def C^(1), schrödinger} 
C_{j k} = C^{(1)}(t_j,t_k) \text{ is the closure of \,} 2 \sum_{\kappa=1}^d (b_{k \, \kappa}-b_{j \, \kappa}) \partial_{\kappa},
\end{align}
that is, $C_{j k}$ generates the translation group $t \mapsto e^{ 2(b_{k \, 1}-b_{j\,1})\partial_1 t} \dotsb e^{ 2(b_{k \, d}-b_{j\,d})\partial_d t}$. 
And from~\eqref{eq: schrödinger, vertrel halbgr}, in turn, the commutation relations imposed in Theorem~\ref{thm: wohlgestelltheit, allg p} follow by differentiation at $\sigma = 0$. Since, moreover, the maximal continuity subspace for $A = C^{(0)}$ and $C^{(1)}$ contains the dense subspace of Schwartz functions on $\R^d$, the existence of a unique evolution system $U$ for $A$ on $Y^{\circ}$ follows by Theorem~\ref{thm: wohlgestelltheit, allg p}.
\smallskip

(ii) We now show the following representation formula for $U$:
\begin{align} \label{eq: darstformel schrödinger, 1}
U(t,s) = W(t) \tilde{U}(t,s) W(s)^{-1} = e^{\ol{(\int_0^t B(\tau) \,d\tau)^{\circ}}} \, e^{\ol{\int_s^t \tilde{A}(\tau) \,d\tau}} \, e^{-\ol{(\int_0^s B(\tau) \,d\tau)^{\circ}}}, 
\end{align}
where $\tilde{U}$ is the evolution system for $\tilde{A}$ on $D := W^{2,2}(\R^d)$ with $\tilde{A}(t) := -i (-i \nabla - c(t))^2$ and $c(t) := \int_0^t b(\tau) \,d\tau$ and where the gauge transformation $W$ is the evolution system for $B$ on $Z^{\circ}$, the maximal continuity subspace for $B$. Clearly, since $B(\tau) = -i b(\tau) \cdot x$ and $\tilde{A}(\tau) = -i \mathcal{F}^{-1} (\xi - c(\tau))^2 \mathcal{F}$, 
\begin{align} \label{eq: darstformel schrödinger, 2}
e^{\ol{(\int_0^t B(\tau) \,d\tau)^{\circ}}} = e^{-i \int_0^t b(\tau) \cdot x \,d\tau} \quad \text{and} \quad 
e^{\ol{\int_s^t \tilde{A}(\tau) \,d\tau}} = \mathcal{F}^{-1} e^{-i \int_s^t (\xi-c(\tau))^2 \, d\tau} \mathcal{F}
\end{align}
(which last expression could be cast into 
a more explicit integral form similar to the explicit integral representation of the free Schrödinger group). 
It should be noticed that, due to the pairwise commutativity of the opertors $\tilde{A}(t)$ and of the operators $B(t)$, the existence of the evolution systems $\tilde{U}$ and $W$, and the second equality in~\eqref{eq: darstformel schrödinger, 1} already follow by~\cite{Goldstein69} and~\cite{NickelSchnaubelt98}. In order to see the first equality in~\eqref{eq: darstformel schrödinger, 1}, one shows by similar arguments as those of part~(i) above 
that the 
subspace $Y_0^{\circ} := D \cap Z^{\circ}$ of $Y^{\circ}$ is invariant under $W(s)^{-1}$, $\tilde{U}(t,s)$, $W(t)$ and that 
\begin{gather*}
A_0 W(t)f = W(t) \tilde{A}(t)f \\ 
B(r) \tilde{U}(t,s)f = \tilde{U}(t,s)\Big( B(r)f - 2 \sum_{\kappa=1}^d b_{\kappa}(r) \, (t-s) \, \partial_{\kappa}f + 2i \sum_{\kappa=1}^d b_{\kappa}(r) \int_s^t c_{\kappa}(\tau)\,d\tau \, f \Big) 
\end{gather*}
for $f \in Y_0^{\circ}$. (Show commutation relations for $e^{\tilde{A}(r_1) \sigma}$ and $e^{B(r_2) \tau}$ analogous to~\eqref{eq: schrödinger, vertrel} to obtain commutation relations for $B(r_2)$ with $e^{\tilde{A}(r_1) \sigma}$ and then use the standard product approximants for the evolution systems $W$ and $\tilde{U}$.)
It then follows that $U_0$ defined by $U_0(t,s) := W(t) \tilde{U}(t,s) W(s)^{-1}$ is an evolution system for $A$ on $Y_0^{\circ}$, which by the standard uniqueness 
argument for evolution systems must coincide with $U$. 
\end{proof}

We see from part~(ii) of the above proof that the existence of an evolution system $U_0$ for $A$ on the subspace $Y_0^{\circ}$, after a suitable gauge transformation, already follows by~\cite{Goldstein69}, \cite{NickelSchnaubelt98} -- but in order to obtain well-posedness on $Y^{\circ}$, the results from~\cite{Goldstein69}, \cite{NickelSchnaubelt98} do not suffice, because the subspace $Y^{\circ}_0$ is strictly contained in $Y^{\circ}$ in general. 
(Indeed, if for instance $b(t) \equiv 1 \in \R^d$ with $d = 1$, then the function $\psi$ with $\psi(\xi) := e^{i \xi^3/3} / \xi$ for $\xi \in [1,\infty)$ and $\psi(\xi) := 0$ for $\xi \in (-\infty,1)$ does not belong to the range of $C-i := i \partial_{\xi} + {\xi}^2 - i$. Consequently, 
$-\partial_x^2 + x - i = \mathcal{F}^{-1} (C-i) \mathcal{F}$ is not surjective so that 
\begin{align*}
Y_0^{\circ} = D(A_0+B) = D(-\partial_x^2 + x) \subsetneq D(\ol{-\partial_x^2 + x}) = D(A) = Y^{\circ}
\end{align*}
by the standard criterion for self-adjointness.) 
We finally remark that the results of~\cite{Yajima87} do not apply to the situation of this section.

\bigskip

Acknowledgement: I would like to thank Prof.~Marcel Griesemer for many helpful discussions and comments. Also I would like to thank DFG for financial support under the grant GR 3213/1-1.

\end{document}